\documentclass[final]{siamltex}

\usepackage{tikz}
\usepackage{array}
\usepackage{amssymb}
\usepackage{color}
\usepackage{epsfig}
\usepackage{graphicx}
\usepackage{makeidx}
\usepackage{multicol}
\usepackage{enumerate}
\usepackage{amsfonts}
\usepackage{amsmath}
\usepackage{amssymb}
\usepackage{amsbsy}
\usepackage{hyperref}

\newcolumntype{V}{>{\raggedright\arraybackslash} m{.47\linewidth}}

\newdimen\pHeight
\newdimen\pLower
\newdimen\pLineWidth
\newdimen\pKern
\newdimen\pIR
\newsavebox{\Cbox}
\newsavebox{\vertCmplx}
\newdimen\Cheight
\newdimen\Cwidth
\sbox{\Cbox}{\rm C}
\sbox{\vertCmplx}{\rule[\pLower]{\pLineWidth}{\Cheight}}
\sbox{\Cbox}{\usebox{\Cbox}\kern\pKern\usebox{\vertCmplx}}
\wd\Cbox=\Cwidth

\def\R{{\rm I\kern\pIR R}}

\mathcode `:="003A
\newcommand{\inv}[2]{\mbox{${#1} \in
\R^{\hspace*{-.01in} #2}$}}  

\newcommand{\inm}[3]
   {\mbox{${#1} \in \R^{\hspace*{-.015in}{#2} \times
{#3}}\hspace*{-.05in}$ }}

\newcommand{\inTens}[4]
   {\mbox{${\cal #1}_{#2} \in \R^{\hspace*{-.015in}{#3} \times \cdots \times {#4}} \hspace*{-.05in}$ }}

\newcommand{\T}{\raisebox{1pt}{$ \:\otimes \:$}}

\mathcode `:="003A

\newenvironment{code}{\begin{tabbing}\hspace*{.3in} \=
\hspace*{.3in} \= \hspace*{.3in} \= \hspace*{.3in} \=
\hspace*{.3in} \= \hspace*{.3in} \= \kill }{\end{tabbing}}

\title{Block Tensor Unfoldings}
\author{Stefan Ragnarsson\thanks{Center for Applied Mathematics, Cornell University 
        Ithaca, NY 14853, {\tt str23@cornell.edu}}
\and Charles F. Van Loan\thanks{Department of Computer Science, Cornell University,
        Ithaca, NY 14853, {\tt cv@cs.cornell.edu}. Both authors
are supported in part by NSF contract DMS-1016284.}
}
\begin{document}
\maketitle
\begin{abstract}
Within the field of numerical multilinear algebra, block tensors are 
increasingly important. 
Accordingly, it is appropriate to develop an infrastructure that supports reasoning about block tensor computation. In this paper we establish concise notation that is suitable for the analysis and development of block tensor algorithms, prove several useful block tensor identities, and make precise the notion of a block tensor unfolding.
\end{abstract}


\pagestyle{myheadings}
\thispagestyle{plain}
\markboth{S. RAGNARSSON and C.F. VAN LOAN}{Block Tensor Unfoldings}

\newcommand{\smallN}{\scriptscriptstyle N}
\newcommand{\smallM}{\scriptscriptstyle M}
\newcommand{\smallB}{\scriptscriptstyle B}
\newcommand{\smallR}{\scriptscriptstyle R}
\newcommand{\smallC}{\scriptscriptstyle C}
\newcommand{\smallD}{\scriptscriptstyle D}
\newcommand{\smallS}{\scriptscriptstyle{S}}
\newcommand{\smallT}{\scriptscriptstyle{T}}
\newcommand{\smallI}{\scriptscriptstyle{I}}
\newcommand{\smallJ}{\scriptscriptstyle{J}}
\newcommand{\smallLambda}{\scriptscriptstyle \Lambda}
\newcommand{\smallPsi}{\scriptscriptstyle \Psi}
\newcommand{\br}{\mathbf{r}}
\newcommand{\bc}{\mathbf{c}}

\newcommand{\boldalpha}{\mbox{\boldmath$\alpha$}}
\newcommand{\boldbeta}{\mbox{\boldmath$\beta$}}
\newcommand{\bolddelta}{\mbox{\boldmath$\delta$}}
\newcommand{\boldlambda}{\mbox{\boldmath$\lambda$}}
\newcommand{\boldpsi}{\mbox{\boldmath$\psi$}}
\newcommand{\boldell}{\mbox{\boldmath$\ell$}}

\section{Introduction}

The field of  matrix computations has matured to the point that it is not necessary
to provide scalar-level verifications of basic block-level operations. For example, if
\[
\left[ \begin{array}{cc} C_{11} & C_{12} \\ C_{21} & C_{22} \rule{0pt}{12pt}\end{array} \right]
\:=\:
\left[ \begin{array}{cc} A_{11} & A_{12} \\ A_{21} & A_{22} \rule{0pt}{12pt}\end{array} \right]^{T}
\left[ \begin{array}{cc} B_{11} & B_{12} \\ B_{21} & B_{22} \rule{0pt}{12pt}\end{array} \right],
\]
then without ``$ijk$ proof'' it is \underline{understood}  that $C_{12} \:=\: A_{11}^{T}B_{12} + A_{21}^{T}B_{22}$ provided $A$ and $B$ are partitioned conformally.
``Understandings'' like this contribute
to
the culture of block matrix computations, enabling researchers to 
think at a high level when they are developing new algorithms and proofs.

It is our contention that the emerging field of tensor computations needs to develop
a similar infrastructure that gracefully supports block tensor operations.
By a {\em block tensor } we mean a tensor whose entries are themselves 
tensors. As with matrices,
the act of blocking a tensor is the act of partitioning the index range vectors
associated with each dimension.
Thus, if \inv{{\cal A}}{9 \times 5 \times 8} and
\[
\begin{array}{lclcl}
1:9 &=& \left[ \begin{array}{cc|ccc|cccc} \!1\!&\!2&3\!&\!4\!&\!5&6\!&\!7\!&\!8\!&\!9\!\end{array} \right]&=&
\left[ \begin{array}{c|c|c} 1:2 & 3:5 & 6:9 \end{array} \right] \\
1:5 &=& \left[ \begin{array}{ccc|cc} \!1\!&\!2\!&3&4\!&\!5\!\end{array} \right] &=&
\left[ \begin{array}{c|c} 1:3 & 4:5  \end{array} \right] \rule{0pt}{13pt}\\
1:8 &=& \left[ \begin{array}{cc|cc|cc|cc} \!1\!&\!2&3\!&\!4&5\!&\!6&7\!&\!8\end{array} \right]&=&\left[ \begin{array}{c|c|c|c} 1:2 & 3:4 & 5:6 & 7:8 \end{array} \right], \rule{0pt}{13pt}
\end{array}
\]
then we are choosing to regard $\cal A$ as a 3-by-2-by-4 block tensor with block dimensions that are determined by the indicated partitionings of $1:9$, $1:5$, and $1:8$. The 
colon notation can be used to specify the blocks. For example, the (2,1,3) block 
${\cal A}_{213}$, is prescribed by
${\cal A}(3:5,1:3,5:6)$.

Block tensors  are  increasingly important for the same reasons that 
block matrices are  increasingly important:
\begin{enumerate}
  \item {\em Structure.} Block-level sparsity is a common pattern because of nearest-neighbor coupling and other reasons \cite{sym}.
  \item {\em Generalization.} Block versions of point algorithms frequently have attractive features \cite{blockPara}.
  \item {\em Performance.} Blocking is the key to minimizing the overhead of communication \cite{MinComm}.
  \end{enumerate} 
Indeed, there is a very strong coupling between block tensor computations and block matrix 
computations. This is because
the dominant  paradigm for tensor computation involves the device of
 {\em unfolding.} 
An unfolded (or flattened) tensor is a matrix obtained by systematically reorganizing  the tensor's entries into a 2-dimensional array. In this framework, computations on a tensor $\cal A$ reduce to
 matrix computations on one or more of its unfoldings. For example, the higher-order singular value decomposition of a tensor involves computing the SVD of each modal unfolding \cite{HOSVD}.
See \cite{TammySirev} for a nice overview of tensor decompositions and unfoldings.

Given all the advantages  that result when a matrix computation is organized at the block level, 
it makes sense for an unfolding of a block tensor $\cal A$ to have a related block structure of
its own. In particular, $\cal A$'s blocks should map to contiguous blocks in the unfolding.
This is not the case when a typical ``vec-oriented'' unfolding is invoked \cite{TammySirev}. Consider the mode-1 unfolding ${\cal A}_{(1)}$ of a 9-by-5-by-8 tensor $\cal A$ with blocking (1.1). The unfolding, which is displayed in Fig 1.1,  is a 9-by-40 matrix
whose $i$-th row is 
$\mbox{vec}({\cal A}(i,:,:))^{T}$. (Recall that vec-of-a-matrix
is the vector obtained by stacking its columns.) 
\begin{figure}[t]
\scalebox{.8}
{
\begin{picture}(300,150)(-20,0)
\put(67,115){\thicklines\vector(-2,-3){40}}
\put(67,115){\thicklines\vector(1,-4){14}}
\put(40,120){\small ${\cal A}_{311} = {\cal A}(6:9,1:3,1:2)$}

\put(210,115){\thicklines\vector(2,-1){64}}
\put(210,115){\thicklines\vector(1,-2){15}}
\put(180,120){\small ${\cal A}_{213} = {\cal A}(3:5,1:3,5:6)$}

\put(348,115){\thicklines\vector(1,-4){3.5}}
\put(348,115){\thicklines\vector(4,-1){50}}
\put(320,120){\small ${\cal A}_{124} = {\cal A}(1:2,4:5,7:8)$}

\put(18,16){\circle*{7}} 
\put(18,26){\circle*{7}} 
\put(18,36){\circle*{7}} 
\put(18,46){\circle*{7}}
\put(28,16){\circle*{7}} 
\put(28,26){\circle*{7}} 
\put(28,36){\circle*{7}} 
\put(28,46){\circle*{7}}
\put(38,16){\circle*{7}} 
\put(38,26){\circle*{7}} 
\put(38,36){\circle*{7}} 
\put(38,46){\circle*{7}}
\put(68,16){\circle*{7}} 
\put(68,26){\circle*{7}} 
\put(68,36){\circle*{7}} 
\put(68,46){\circle*{7}}
\put(78,16){\circle*{7}} 
\put(78,26){\circle*{7}} 
\put(78,36){\circle*{7}} 
\put(78,46){\circle*{7}}
\put(88,16){\circle*{7}} 
\put(88,26){\circle*{7}} 
\put(88,36){\circle*{7}} 
\put(88,46){\circle*{7}}

\put(218,75){\circle*{7}}
\put(228,75){\circle*{7}}
\put(238,75){\circle*{7}}
\put(268,75){\circle*{7}}
\put(278,75){\circle*{7}}
\put(288,75){\circle*{7}}
\put(218,65){\circle*{7}}
\put(228,65){\circle*{7}}
\put(238,65){\circle*{7}}
\put(268,65){\circle*{7}}
\put(278,65){\circle*{7}}
\put(288,65){\circle*{7}}
\put(218,55){\circle*{7}}
\put(228,55){\circle*{7}}
\put(238,55){\circle*{7}}
\put(268,55){\circle*{7}}
\put(278,55){\circle*{7}}
\put(288,55){\circle*{7}}

\put(348,85){\circle*{7}}
\put(348,95){\circle*{7}}
\put(358,85){\circle*{7}}
\put(358,95){\circle*{7}}
\put(398,85){\circle*{7}}
\put(398,95){\circle*{7}}
\put(408,85){\circle*{7}}
\put(408,95){\circle*{7}}

\put(10,10){   \line(1,0){400}}
\put(10,20){   \line(1,0){400}}
\put(10,30){  \line(1,0){400}}
\put(10,40){  \line(1,0){400}}
\put(10,50){  \line(1,0){400}}
\put(10,60){  \line(1,0){400}}
\put(10,70){  \line(1,0){400}}
\put(10,80){  \line(1,0){400}}
\put(10,90){  \line(1,0){400}}
\put(10,100){  \line(1,0){400}}
\put(10,10){  \line(0,1){90}}
\put(20,10){  \line(0,1){90}}
\put(30,10){  \line(0,1){90}}
\put(40,10){  \line(0,1){90}}
\put(50,10){  \line(0,1){90}}
\put(60,10){  \line(0,1){90}}
\put(70,10){  \line(0,1){90}}
\put(80,10){  \line(0,1){90}}
\put(90,10){  \line(0,1){90}}
\put(100,10){  \line(0,1){90}}
\put(110,10){  \line(0,1){90}}
\put(120,10){  \line(0,1){90}}
\put(130,10){  \line(0,1){90}}
\put(140,10){  \line(0,1){90}}
\put(150,10){  \line(0,1){90}}
\put(160,10){  \line(0,1){90}}
\put(170,10){  \line(0,1){90}}
\put(180,10){  \line(0,1){90}}
\put(190,10){  \line(0,1){90}}
\put(200,10){  \line(0,1){90}}
\put(210,10){  \line(0,1){90}}
\put(220,10){  \line(0,1){90}}
\put(230,10){  \line(0,1){90}}
\put(240,10){  \line(0,1){90}}
\put(250,10){  \line(0,1){90}}
\put(260,10){  \line(0,1){90}}
\put(270,10){  \line(0,1){90}}
\put(280,10){  \line(0,1){90}}
\put(290,10){  \line(0,1){90}}
\put(300,10){  \line(0,1){90}}
\put(310,10){  \line(0,1){90}}
\put(320,10){  \line(0,1){90}}
\put(330,10){  \line(0,1){90}}
\put(340,10){  \line(0,1){90}}
\put(350,10){  \line(0,1){90}}
\put(360,10){  \line(0,1){90}}
\put(370,10){  \line(0,1){90}}
\put(380,10){  \line(0,1){90}}
\put(390,10){  \line(0,1){90}}
\put(400,10){  \line(0,1){90}}
\put(410,10){  \line(0,1){90}}
\end{picture}
}
\caption{A vec-ordered, mode-1 unfolding of \inv{\cal A}{9\times 5 \times 8} with blocking (1.1)}
\end{figure}
Notice that in the unfolding, $\cal A$'s flattened blocks are not contiguous.
The primary purpose of this paper is to show how to permute the rows and columns of a vec-oriented unfolding so that its blocks are unfoldings  of the tensor blocks.
An example of such an unfolding is displayed in Fig 1.2. 
\begin{figure}
\scalebox{.8}
{
\begin{picture}(300,120)(-20,-10)
\put(-10,87){\large 2 $\left\{ \rule{0pt}{10pt}\right.$}
\put(-10,63){\large 3 $\left\{ \rule{0pt}{12pt}\right.$}
\put(-10,27){\large 4 $\left\{ \rule{0pt}{16pt}\right.$}

\put(18,4){$\underbrace{\rule{50pt}{0pt}}$}
\put(41,-14){\large 6}
\put(78,4){$\underbrace{\rule{30pt}{0pt}}$}
\put(90,-14){\large 4}

\put(118,4){$\underbrace{\rule{50pt}{0pt}}$}
\put(141,-14){\large 6}
\put(178,4){$\underbrace{\rule{30pt}{0pt}}$}
\put(190,-14){\large 4}

\put(218,4){$\underbrace{\rule{50pt}{0pt}}$}
\put(241,-14){\large 6}
\put(278,4){$\underbrace{\rule{30pt}{0pt}}$}
\put(290,-14){\large 4}

\put(318,4){$\underbrace{\rule{50pt}{0pt}}$}
\put(341,-14){\large 6}
\put(378,4){$\underbrace{\rule{30pt}{0pt}}$}
\put(390,-14){\large 4}

\put(22,87){ ${(\cal A}_{111})_{_{\scriptscriptstyle(1)}}$}
\put(22,62){ ${(\cal A}_{211})_{_{\scriptscriptstyle(1)}}$}
\put(22,28){ ${(\cal A}_{311})_{_{\scriptscriptstyle(1)}}$}

\put(71,87){ ${(\cal A}_{121})_{_{\scriptscriptstyle(1)}}$}
\put(71,62){ ${(\cal A}_{221})_{_{\scriptscriptstyle(1)}}$}
\put(71,28){ ${(\cal A}_{321})_{_{\scriptscriptstyle(1)}}$}

\put(122,87){$({\cal A}_{112})_{_{\scriptscriptstyle(1)}}$}
\put(122,62){$({\cal A}_{212})_{_{\scriptscriptstyle(1)}}$}
\put(122,28){$({\cal A}_{312})_{_{\scriptscriptstyle(1)}}$}

\put(171,87){ $({\cal A}_{122})_{_{\scriptscriptstyle(1)}}$}
\put(171,62){ $({\cal A}_{222})_{_{\scriptscriptstyle(1)}}$}
\put(171,28){ $({\cal A}_{322})_{_{\scriptscriptstyle(1)}}$}

\put(222,87){ $({\cal A}_{113})_{_{\scriptscriptstyle(1)}}$}
\put(222,62){ $({\cal A}_{213})_{_{\scriptscriptstyle(1)}}$}
\put(222,28){ $({\cal A}_{313})_{_{\scriptscriptstyle(1)}}$}

\put(271,87){ $({\cal A}_{123})_{_{\scriptscriptstyle(1)}}$}
\put(271,62){ $({\cal A}_{223})_{_{\scriptscriptstyle(1)}}$}
\put(271,28){ $({\cal A}_{323})_{_{\scriptscriptstyle(1)}}$}

\put(322,87){$({\cal A}_{114})_{_{\scriptscriptstyle(1)}}$}
\put(322,62){$({\cal A}_{214})_{_{\scriptscriptstyle(1)}}$}
\put(322,28){$({\cal A}_{314})_{_{\scriptscriptstyle(1)}}$}

\put(374,87){$({\cal A}_{124})_{_{\scriptscriptstyle(1)}}$}
\put(374,62){$({\cal A}_{224})_{_{\scriptscriptstyle(1)}}$}
\put(374,28){$({\cal A}_{324})_{_{\scriptscriptstyle(1)}}$}

\put(10,10){   \line(1,0){400}}

\put(10,50){  \line(1,0){400}}
\put(10,80){  \line(1,0){400}}
\put(10,100){  \line(1,0){400}}

\put(10,10){  \line(0,1){90}}
\put(70,10){  \line(0,1){90}}
\put(110,10){  \line(0,1){90}}
\put(170,10){  \line(0,1){90}}
\put(210,10){  \line(0,1){90}}
\put(270,10){  \line(0,1){90}}
\put(310,10){  \line(0,1){90}}
\put(370,10){  \line(0,1){90}}
\put(410,10){  \line(0,1){90}}
\end{picture}
}
\caption{A ``block vec''-ordered, mode-1 unfolding of \inv{\cal A}{9\times 5 \times 8} with blocking (1.1)}
\end{figure}

The paper is organized as follows. In \S2 we review well-known connections between  $\mbox{vec}(\cdot)$, Kronecker products,  transposition, and the perfect shuffle permutation. A block version of $\mbox{vec}(\cdot)$ is defined in \S3 and a related permutation is used to define the notion of a block unfolding.
In \S4 we show how to formulate a tensor contraction as a block matrix multiplication
using the tools developed.

\section{Basic Notation and Operations
}
If \inTens{A}{}{n_{1}}{n_{d}} and $\mathbf{i} = (i_{1},\ldots,i_{d})$, then 
${\cal A}(\mathbf{i})$ denotes component $(i_{1},\ldots,i_{d})$ of tensor $\cal A$.
We use calligraphic characters to designate tensors and bold lower case characters to
denote  vectors of integers. For ${\cal A}(\mathbf{i})$ to make sense we must have
$1\leq i_{k} \leq n_{k}$ for $k=1:d$, i.e., $\mathbf{1} \leq \mathbf{i} \leq \mathbf{n}$.
In general,  if
$\mathbf{i}$ and $\mathbf{j}$ have equal length, then $\mathbf{i} \leq \mathbf{j}$ means
that $i_{k} \leq j_{k}$ for all $k$.

The \textsc{Matlab} colon notation is used to specify index ranges.
If $a<b$ and $c>0$, then  $a:b$ is the vector $[\:a,a+1,\ldots,b\:]$ and $a:c:b$ is the vector $[\:a,a+c,a+2c,\ldots,a+mc\:]$ where $m = \lfloor(b-a)/c\rfloor$, i.e.~the largest integer that is less than or equal to $(b-a)/c$.

If $A\in \mathbb{R}^{m\times n}$ and $B \in \mathbb{R}^{p\times q}$, then the Kronecker product $A \otimes B \in \mathbb{R}^{mp\times nq}$ is the block matrix
\[
A \otimes B = \left[ \begin{array}{c c c}
a_{11}B & \cdots & a_{1n}B \\
\vdots & \ddots & \vdots \\
a_{m1}B & \cdots & a_{mn}B 
\end{array}\right].
\]
The outer product ${\cal C} = {\cal A} \circ {\cal B}$ of a tensor \inv{\cal A}{j_{1} \times \cdots \times j_{d}}
and a tensor \inv{\cal B}{k_{1} \times \cdots \times k_{e}} is a 
tensor \inv{\cal C}{j_{1} \times \cdots \times j_{d}\times k_{1} \times \cdots \times k_{e}}
defined by
\[
{\cal C}(\mathbf{i}) = {\cal A} (\,\mathbf{i}(1:d)\,) \cdot \mathcal{B} (\,\mathbf{i}(d+1 : d+e)\,)
\qquad  \mathbf{1}\leq \mathbf{i} \leq [\: \mathbf{j} \; \mathbf{k}\: ].
\]
The order of ${\cal A} \circ {\cal B}$ is the order of $\cal A$ plus the order of $\cal B$.
Note that $A \T B$ is an unfolding of the order-4 tensor ${\cal A}\circ {\cal B}$ where
${\cal A}$ and $\cal B$ are order-2 tensors (matrices) $A$ and $B$.

\subsection{The Vec Operation and Ordering}

If \inTens{A}{}{n_{1}}{n_{d}} and $N = n_{1}\cdots n_{d}$, then  
\inv{\mbox{vec}({\cal A})}{\smallN} is a column vector defined recursively by  
\begin{equation}
\mbox{vec}({\cal A}) \;=\; \left[ \begin{array}{c}
\mbox{vec}({\cal A}^{(1)}) \\ 
\vdots \\
\mbox{vec}({\cal A}^{(n_{d})}) \end{array} \right]
\end{equation}
where ${\cal A}^{(k)}$ is the order-$(d-1)$ tensor 
\begin{equation}
{\cal A}^{(k)}(i_{1},\ldots,i_{d-1}) \; =\;{\cal A}(i_{1},\ldots,i_{d-1},k) \qquad 1\leq k \leq n_{d} .
\end{equation}
It is assumed that $\mathbf{1} \leq \mathbf{i}(1:d-1) \leq \mathbf{n}(1:d-1).$
If $d =1 $, then $\cal A$ is a column vector and $\mbox{vec}({\cal A}) = {\cal A}$.
If $d=2$, then $\cal A$ is a matrix and $\mbox{vec}({\cal A})$  stacks its columns.
Each entry in tensor \inTens{A}{}{n_{1}}{n_{d}} corresponds to a component of 
vec$(\cal A )$.
This implicitly defines an index mapping function
$ivec(\,\cdot\,,\mathbf{n})$:
\begin{equation}
ivec(\mathbf{i},\mathbf{n}) \;=\;
i_{1} + (i_{2}-1)n_{1} + (i_{3}-1)n_{1}n_{2} +\cdots + (i_{d}-1) n_{1}\cdots n_{d-1}.
\end{equation}
It is easy to show that if $v = \mbox{vec}({\cal A})$, then
\begin{equation}
 v_{ivec(\mathbf{i},\mathbf{n})} \;=\; {\cal A}(\mathbf{i})
\end{equation}
for all $\mathbf{i}$ that satisfy $\mathbf{1} \leq \mathbf{i} \leq \mathbf{n}$. 

It should be noted that the ``tensor vec'' operation given by (2.1)-(2.4) reverts
to the standard vec operation when $\cal A$ is a matrix \cite{Searle}.

\subsection{Transposition, Vec, Kronecker Products, and Permutation}

There is an important connection between
matrix transposition and perfect shuffle permutations \cite{Searle,HornerJohnson,cvFFT,cvKron}. In particular, if \inm{A}{q}{r}
and $s = qr$,
then
\begin{equation}
\mbox{vec}(A^{T}) \;=\; \Pi_{q,r}^{T} \mbox{vec}(A)
\end{equation}
where \inm{\Pi_{q,r}}{s}{s} is the $(q,r)$ \emph{perfect shuffle} permutation 
defined by
\begin{equation}
\Pi_{q,r} z \;=\;\left[ \begin{array}{c} z(1:r:s) \\ z(2:r:s) \\ \vdots \\ z(r:r:s)
\end{array} \right] \qquad \inv{z}{s}.
\end{equation}
See \cite{cvFFT}.
 If \inm{Z}{r}{q} and $Y = Z^{T}$, then $\mbox{vec}(Y) = \Pi_{q,r}\mbox{vec}(Z)$. It is easy to verify that $\Pi_{q,r}^{T} = \Pi_{r,q}$.

If 
\inv{f}{q} and  \inv{g}{r}, then $g \T f$ is a perfect shuffle of $f \T g$:
\begin{equation}
 \Pi_{q,r} \left( f \T g \right)  \;=\; g \T f.
\end{equation}
An important consequence of this result applies to the case when $g$ is a block vector:
\begin{equation}
\mbox{diag}(\Pi_{\rho_{1},q},\ldots,\Pi_{\rho_{\mu},q})\cdot \Pi_{q,r} \cdot
\left(
f\T \left[ \begin{array}{c} g_{1} \\ \vdots \\ g_{\mu} \end{array} \right]\right)
\;=\;
\left[ \begin{array}{c} f\T g_{1} \\ \vdots \\ f\T g_{\mu} \end{array} \right].
\end{equation}
Here, \inv{g_{i}}{\rho_{i}} and $r = \rho_{1}+\cdots + \rho_{\mu}$.

Tensor transposition can also be characterized in terms of $\mbox{vec}(\cdot )$ and perfect shuffles.
If 
$ \inTens{A}{}{n_{1}}{n_{d}}$ and $\mathbf{p}$
is a permutation of $1:d$, then 
\inv{{\cal A}^{<\mathbf{p}>}}{n_{p_{1}}\times \cdots \times n_{p_{d}}}
denotes the
$\mathbf{p}$-transpose of $\cal A$ and is defined by
\begin{equation}
{\cal A}^{<\mathbf{p}>}(i_{p_{1}},\ldots,i_{p_{d}}) \;=\; {\cal A}(i_{1},\ldots,i_{d}) 
\qquad \mathbf{1} \leq \mathbf{i} \leq \mathbf{n},
\end{equation}
i.e., ${\cal A}^{<\mathbf{p}>}(\mathbf{i}(\mathbf{p})) = {\cal A}(\mathbf{i})$.
The following lemma can be regarded as a generalization of (2.5):

\medskip

\begin{lemma}
If \inv{\cal A}{\smallN_{1}\times \smallN_{2} \times \smallN_{3} \times \smallN_{4}} and 
${\cal B} = {\cal A}^{<[\:1\:3\:2\:4]>}$, then
\[
\mbox{\rm vec}({\cal B}) \;=\; 
(I_{\smallN_{4}} \T \Pi_{\smallN_{3},\smallN_{2}} \T I_{\smallN_{1}})
\mbox{\rm vec}({\cal A}).
\]
\end{lemma}
\begin{proof}
The proof follows from well-known facts that relate Kronecker products, $\mbox{vec}(\cdot )$, and the perfect shuffle. See  \cite{Searle, HornerJohnson, cvFFT}.
\end{proof}

\medskip

\noindent
Although Lemma 2.1 addresses an order-4 transposition, the result can be applied
to tensors  of arbitrary order simply by ``fusing'' adjacent modes. For example, 
suppose \inTens{C}{}{n_{1}}{n_{7}} and set  $N_{1} = n_{1}n_{2}$,
$N_{2} = n_{3}$, $N_{3} = n_{4}n_{5}$, and $N_{4} = n_{6}n_{7}$. Define \inv{\cal A}{\smallN_{1} \times \smallN_{2} \times \smallN_{3} \times \smallN_{4}}
by
\[
{\cal A}(j_{1},j_{2},j_{3},j_{4})\:=\:{\cal C}(\mathbf{i}) \quad \mbox{where} \quad \left\{
\begin{array}{lcl}
j_{1} &=& ivec(\mathbf{i}(1:2),\mathbf{n}(1:2))\\
j_{2} &=& ivec(\mathbf{i}(3:3),\mathbf{n}(3:3)) \rule{0pt}{12pt}\\
j_{3} &=& ivec(\mathbf{i}(4:5),\mathbf{n}(4:5)) \rule{0pt}{12pt}\\
j_{4} &=& ivec(\mathbf{i}(6:7),\mathbf{n}(6:7)) \rule{0pt}{12pt}
\end{array}.
\right.
\]
Observe that $\mbox{vec}({\cal A}) = \mbox{vec}({\cal C})$ and 
\begin{eqnarray*}
(I_{\smallN_{4}} \T \Pi_{\smallN_{3},\smallN_{2}} \T I_{\smallN_{1}})
\mbox{\rm vec}({\cal C}) &=&
(I_{\smallN_{4}} \T \Pi_{\smallN_{3},\smallN_{2}} \T I_{\smallN_{1}})
\mbox{\rm vec}({\cal A})\\
&=&  \mbox{vec}({\cal A}^{<[\:1\: 3 \:2 \:4\:]>})  
\;=\;  \mbox{vec}({\cal C}^{<[\:1\: 2 \:4 \:5\:3\:6\:7\:]>}) . \rule{0pt}{12pt}
\end{eqnarray*}

Two special applications of Lemma 2.1 are worth  noting.
Assume  \inTens{A}{}{n_{1}}{n_{d}}. If
$\mathbf{p} \;=\; [ \:1:k\!-\!1 \:, \: k\!+\!1 \:, \:k\:,\: k\!+\!2:d\:]$,
then
\begin{equation}
\mbox{\rm vec}({\cal A}^{<\mathbf{p}>}) \;=\; (I_{\smallN_{4}} \T \Pi_{n_{k+1},n_{k}} \T I_{\smallN_{1}})
\mbox{\rm vec}({\cal A})
\end{equation}
where $N_{1} = n_{1}\cdots n_{k\!-\!1}$ and $N_{4} = n_{k+2}\cdots n_{d}$. 
This transposition swaps  two adjacent modes, e.g.,
\[
{\cal B} = {\cal A}^{<\:[1\:2\:4\:3\:5\:]\:>} \Rightarrow
{\cal A}(i_{1},i_{2},i_{3},i_{4},i_{5})\; = \;
{\cal B}(i_{1},i_{2},i_{4},i_{3},i_{5}).
\]
On the other hand, 
if  
$\mathbf{p} \;=\; [ \: k \:,\: 1:k-1 \:,\: k+1:d\:]$,
then
\begin{equation}
\mbox{\rm vec}({\cal A}^{<\mathbf{p}>}) \;=\; 
(I_{\smallN_{4}} \T \Pi_{N_{2},n_{k}})
\mbox{\rm vec}({\cal A})
\end{equation}
where $N_{2} = n_{1}\cdots n_{k-1}$ and $N_{4} = n_{k+1}\cdots n_{d}$.
This transposition ``moves'' a designated mode ``to the front,'' e.g.,
\[
{\cal B} = {\cal A}^{<\:[\:3\: 1\:2\:4\:5\:]\:>} \Rightarrow
{\cal A}(i_{1},i_{2},i_{3},i_{4},i_{5})\; = \;
{\cal B}(i_{3},i_{1},i_{2},i_{4},i_{5}).
\]

\subsection{Unfolding a Tensor}
Converting a tensor to a matrix is an important operation in tensor computations \cite{Kiers, KoldaMLOps, koldaToolbox,TammySirev}.
In order to unfold a tensor \inTens{A}{}{n_{1}}{n_{d}} into a matrix, it is  necessary to choose (a) an integer
$e$ that satisfies $1\!\leq\! e \!< \!d$ and  (b) a permutation
$\mathbf{p}$ of $1:d$. 
If
\begin{eqnarray}
\mathbf{r} &=& \mathbf{p}(1:e)\\
\mathbf{c} &=& \mathbf{p}(e+1:d) \rule{0pt}{14pt}
\end{eqnarray}
then the  $\mathbf{r}\times \mathbf{c}$ {\em unfolding} of $\cal A$ is the matrix
${\cal A}_{\mathbf{r}\times \mathbf{c}}$ whose $(\alpha,\beta)$ entry is
given by
\begin{equation}
{\cal A}_{\mathbf{r}\times \mathbf{c}}(\alpha,\beta)
 \;=\; {\cal A}^{<\mathbf{p}>}(i_{1},\ldots,i_{e},j_{1},\ldots,j_{d-e})
\end{equation}
where  
\begin{eqnarray}
\alpha &=& ivec(\mathbf{i},\mathbf{n}(\mathbf{r})) \qquad \mathbf{1} \leq \mathbf{i} \leq \mathbf{n}(\mathbf{r})\\
\beta &=&  ivec(\mathbf{j},\mathbf{n}(\mathbf{c})) \qquad \mathbf{1} \leq \mathbf{j} \leq \mathbf{n}(\mathbf{c}).
\rule{0pt}{13pt}
\end{eqnarray}

Note that ${\cal A}_{\mathbf{r}\times \mathbf{c}}$ has $n_{p_{1}}\cdots n_{p_{e}}$ rows
and $n_{p_{e+1}}\cdots n_{p_{d}}$ columns.
 Each row and column of  
${\cal A}_{\mathbf{r}\times \mathbf{c}}$ is the vec of
a reduced-order subtensor.
In particular, for all $\mathbf{i}$ and $\mathbf{j}$ that satisfy
$\mathbf{1} \leq \mathbf{i} \leq \mathbf{n}(\mathbf{r})$ and
$\mathbf{1} \leq \mathbf{j} \leq \mathbf{n}(\mathbf{c})$, we have
\begin{eqnarray}
{\cal A}_{\mathbf{r}\times \mathbf{c}}(\: ivec(\mathbf{i},\mathbf{n}(\mathbf{r})),\::\:) &=&
\mbox{vec}(\,{\cal R}^{(\mathbf{i})})^{T}\\
{\cal A}_{\mathbf{r}\times \mathbf{c}}(\::\:,ivec(\mathbf{j},\mathbf{n}(\mathbf{c}))\:) &=&
\mbox{vec}(\,{\cal C}^{(\mathbf{j})}) \rule{0pt}{14pt}
\end{eqnarray}
where the tensors ${\cal R}^{(\mathbf{i})}$  and   ${\cal C}^{(\mathbf{j})}$ are
defined by
\begin{eqnarray}
R^{(\mathbf{i})}(\mathbf{j}) &=& {\cal A}^{<\mathbf{p}>}(i_{1},\ldots,i_{e},j_{1},\ldots,j_{d-e})\\
C^{(\mathbf{j})}(\mathbf{i}) &=& {\cal A}^{<\mathbf{p}>}(i_{1},\ldots,i_{e},j_{1},\ldots,j_{d-e})
\rule{0pt}{14pt}.
\end{eqnarray}
Especially important are the modal unfoldings.
If $\mathbf{p} = [\:k \;\: 1:k\!-\!1\;\; k\!+\!1:d\:]$, then 
${\cal A}_{\mathbf{r}\times \mathbf{c}}$ is a {\em mode-$k$ unfolding} of $\cal A$. The columns of this
matrix are referred to as  {\em mode-$e$ fibers} of $\cal A$.  
Special conventions are required if $\cal A$ is to be unfolded to either
a column or row vector. 
If $e = d$, then $\mathbf{c}=\emptyset$ and
${\cal A}_{\mathbf{r}\times\mathbf{c}} = \mbox{vec}({\cal A})$. Likewise,
if $e = 0$, then $\mathbf{r}=\emptyset$ and
${\cal A}_{\mathbf{r}\times\mathbf{c}} = \mbox{vec}({\cal A})^{T}$.

\subsection{Special Cases}

The preceding results take on a special form when $\cal A$ is a rank-1 tensor. 
Suppose
${\cal A} \;=\; a^{(1)} \circ \cdots \circ a^{(d)}$ where \inv{a^{(k)}}{n_{k}} for $k=1,
\ldots,d$,
i.e.,
\vspace*{-.0in}
\[ 
{\cal A}(i_{1},\ldots, i_{d}) \;=\; a^{(1)}(i_{1})\cdots a^{(d)}(i_{d})
\qquad \mathbf{1} \leq \mathbf{i} \leq \mathbf{n}.
\]
It follows from (2.1)-(2.4) that
if 
\[
v \;=\; \mbox{vec}(a^{(1)} \circ \cdots \circ a^{(d)}),
\]

\vspace*{-.1in}

\noindent
then
\begin{equation}
v \;=\; a^{(d)} \T \cdots \T a^{(1)}
\end{equation}

\vspace*{-.15in}

\noindent
and

\vspace*{-.2in}

\begin{equation}
v_{ivec(\mathbf{i},\mathbf{n})} \;=\; a^{(1)}(i_{1}) \cdots a^{(d)}(i_{d})
\qquad \mathbf{1} \leq \mathbf{i} \leq \mathbf{n}. \rule{0pt}{15pt}
\end{equation}
If $\mathbf{p}$ is a permutation of $1:d$, then from the 
definition of the $\mathbf{p}$-transpose in (2.9) and the definition
of ${\cal A}_{\mathbf{r}\times \mathbf{c}}$ in (2.12)-(2.16) we have
\begin{equation}
{\cal A}^{<\mathbf{p}>} \;=\; a^{(p_{1})} \circ \cdots \circ a^{(p_{d})}
\end{equation}

\vspace*{-.05in}

\noindent
and

\vspace*{-.2in}
\begin{equation}
{\cal A}_{\mathbf{r}\times \mathbf{c}} \;=\;
\mbox{vec}(a^{(r_{1})} \circ \cdots \circ a^{(r_{e})}) \:\cdot \: \mbox{vec}(a^{(c_{1})} \circ \cdots 
\circ a^{(c_{d-e})})^{T}.
\end{equation}
In other words, the unfolding of a rank-1 tensor is a rank-1 matrix.
These rank-1 facts simplify some of the proofs that follow in the next section.

We consider another special case that relates to the multilinear product, see \S4.2. Suppose ${\cal B} \;= B^{(1)} \circ \cdots \circ B^{({d}\,)}$ where \inm{B^{(k)}}{q_{k}}{n_{k}} for 
$k=1,\ldots,d$, i.e.,
\[
{\cal B}(i_{1},j_{1},\ldots,i_{{d}},j_{{d}}\,)\;=\;
B^{(1)}(i_{1},j_{1})\cdots B^{({d}\,)}(i_{{d}},j_{{d}}).
\]
Note that $\cal B$ is an order-$2d$ tensor. If
$\mathbf{r} \:=\: 1:2:2d$,
$\mathbf{c} \:=\: 2:2:2d$, and
$\mathbf{p} \:=\: [\:\mathbf{r} \; \mathbf{c}\:]$, then
for
all $\mathbf{i}$ and $\mathbf{j}$ that satisfy $\mathbf{1} \leq \mathbf{i} \leq \mathbf{q}$
and $\mathbf{1} \leq \mathbf{j} \leq \mathbf{n}$ we have
\[   
{\cal B}_{\mathbf{r}\times \mathbf{c}}(\alpha,\beta) \;=\;
B^{(1)}(i_{1},j_{1})\cdots B^{({d})}(i_{{d}},j_{{d}})
\]
where $\alpha = ivec(\mathbf{i},\mathbf{q})$ and $\beta = ivec(\mathbf{j},\mathbf{n})$. 
However, this is precisely the $(\alpha,\beta)$ entry of the matrix
$B^{({d}\,)} \T \cdots \T B^{(1)}$. Thus,
\begin{equation}
\left( B^{(1)} \circ \cdots \circ B^{({d}\,)} \right)_{[\,1:2:2d\,]\times[\,2:2:2d\,]}
\;=\;  B^{({d}\,)} \T \cdots \T B^{(1)}.
\end{equation}

\section{Block Notation and Operations}
In this section we formalize the notion of a block tensor \cite{sym}, develop a block version of $\mbox{vec}(\cdot )$, and
explain how to permute ${\cal A}_{\mathbf{r}\times \mathbf{c}}$ into a block matrix whose
blocks are unfoldings of $\cal A$'s blocks. The presentation is simplified if we make use of 
multi-indexed subscripts. Suppose 
\[
\begin{array}{lcl}
\mathbf{1} \leq \mathbf{i} \leq \mathbf{s} = [s_{1},\ldots,s_{e}] & & S = s_{1}\cdots s_{e}\\
\mathbf{1} \leq \mathbf{j} \leq \mathbf{t} = [t_{1},\ldots,t_{f}] & & T = t_{1}\cdots t_{f}
\rule{0pt}{14pt}
\end{array}
\]
To say that $v_{\mathbf{i}}$ is the $\mathbf{i}$-th component
of vector \inv{v}{\smallS} is to say that $v_{\mathbf{i}} = v_{ivec(\mathbf{i},\mathbf{s})}$.
Similarly, if $D_{1},\ldots,D_{\smallS}$ are square matrices and 
$D = \mbox{diag}(\ldots,D_{\mathbf{i}},\ldots)$,
then $D$ is a block diagonal matrix whose $\mathbf{i}$-th diagonal block is 
$D_{ivec(\mathbf{i},\mathbf{s})}.$  Finally, if $C = (C_{ij})$ is an 
$S$-by-$T$ block matrix, then $C_{\mathbf{i},\mathbf{j}}$ is its $(\mathbf{i},\mathbf{j})$-th block, i.e.,
$C_{\mathbf{i},\mathbf{j}} = C_{ivec(\mathbf{i},\mathbf{s}),ivec(\mathbf{j},\mathbf{t})}$.

\subsection{Tensor Blockings}
We say that 
\begin{equation}
\mathbf{M} \;=\; \{\mathbf{m}^{(1)},\ldots,\mathbf{m}^{(d)}\}
\end{equation}
is a {\em blocking} for \inTens{A}{}{n_{1}}{n_{d}} if 
\begin{equation}
\mathbf{m}^{(k)} \;=\; \left[ \:{m}_{1}^{(k)},\ldots,{m}_{b_{k}}^{(k)} \: \right]
\end{equation}
is a vector of positive integers that sums to $n_{k}$ for $k=1,\ldots, d$. 
If $\mathbf{1} \leq \mathbf{i} \leq \mathbf{b}$,
then block  $\mathbf{i}$ is the ${m}_{i_{1}}^{(1)} \times \cdots \times
{m}_{i_{d}}^{(d)}$ tensor 
defined by
\begin{equation}
{\cal A}_{\mathbf{i}} \;=\; {\cal A}(\:\ell_{i_{1}}^{(1)}:u_{i_{1}}^{(1)},\ldots,
\ell_{i_{d}}^{(d)}:u_{i_{d}}^{(d)}\:)
\end{equation}
where the lower and upper bound vectors $\boldell^{(1)},\ldots,\boldell^{(d)}$ and
$\mathbf{u}^{(1)},\ldots,\mathbf{u}^{(d)}$ are defined by
\begin{eqnarray}
\ell^{(k)}_{j} &=& m^{(k)}_{1} + \cdots + m^{(k)}_{j-1}\;+\; 1\\
u^{(k)}_{j} &=& m^{(k)}_{1} + \cdots + m^{(k)}_{j-1}\;+\;
m^{(k)}_{j}
\end{eqnarray}
for $k=1,\ldots,d$. 
The blocking $\mathbf{M}$ identifies $\cal A$
as a $b_{1} \times b_{2} \times \cdots \times b_{d}$ block tensor. The number of
elements in each tensor block ${\cal A}_{\mathbf{i}}$ turns out to be a quantity of importance and to that
end we define the ``volume function'' 
$\mbox{vol}_{\mathbf{\smallM}}(\cdot)$ by
\begin{equation}
\mbox{vol}_{\mathbf{\smallM}}(\mathbf{i}) \;=\; m^{(1)}_{i_{1}} \cdots m^{(d)}_{i_{d}}
\qquad
\mathbf{1} \leq \mathbf{i} \leq \mathbf{b}.
\end{equation}

\subsection{The $\mbox{\bf Vec}_{\mathbf{M}}(\cdot )$ Operation}
If $\mathbf{M}$ is a blocking of \inTens{A}{}{n_{1}}{n_{d}} given by (3.1)-(3.5),
then $\mbox{vec}_{\scriptscriptstyle\mathbf{M}}({\cal A})$ is the block vector
\begin{equation}
\mbox{vec}_{\scriptscriptstyle\mathbf{M}}({\cal A}) \;=\;
\left[ \begin{array}{c} v_{\mathbf{1}} \\ \vdots \\ v_{\mathbf{b}} \end{array} \right]
\qquad v_{\mathbf{i}} \;=\; \mbox{vec}({\cal A}_{\mathbf{i}})
\end{equation}
where $\mathbf{1} \leq \mathbf{i} \leq \mathbf{b}$. In other words, $\mbox{vec}_{\mathbf{\smallM}}({\cal A})$ stacks the vec's of ${\cal A}$'s blocks where the blocks are taken in the vec-order.

To illustrate this notation in the familiar matrix case, if
\[
\mathbf{M} \;=\; \{ {m}^{(1)}, {m}^{(2)}\} \;=\; \{ \:
[\: {m}^{(1)}_{1} \: {m}^{(1)}_{2} \:]\:,\:
[\: {m}^{(2)}_{1} \: {m}^{(2)}_{2} \: {m}^{(2)}_{3} \:]\;\} 
\]
is a blocking for \inm{\cal A}{n_{1}}{n_{2}}, then we are choosing to regard $\cal A$ as 
a 2-by-3 block matrix
\begin{equation}
{\cal A} \;=\; \left[ \begin{array}{ccc}  {\cal A}_{11} & \: {\cal A}_{12}&\: {\cal A}_{13} \\
{\cal A}_{21} & \:{\cal A}_{22} & \:{\cal A}_{23} \rule{0pt}{15pt}\end{array} \right]
\!\! \!
\mbox{\small $\begin{array}{l} \scriptstyle{\;{m}^{(1)}_{1}} \\ 
\scriptstyle{\;{m}^{(1)}_{2}}\rule{0pt}{15pt}\!
\end{array}$} .
\end{equation}
\vspace*{-.25in}
\[
\mbox{\small$\!\!
\begin{array}{lll}\;\;
\scriptstyle{{m}^{(2)}_{1}} \rule{0pt}{12pt}& \;\:\scriptstyle{{m}^{(2)}_{2}} & \;\scriptstyle{{m}^{(2)}_{3}}
\end{array}$}
\]
In this case, $\mbox{vec}_{\scriptscriptstyle\mathbf{M}}(\cdot ) $ and
$\mbox{vol}_{\mathbf{\smallM}}(\cdot)$ are given by
\[
\mbox{vec}_{\scriptscriptstyle\mathbf{M}}(A) \;=\;
\left[ \begin{array}{c}
v_{[1,1]} \\
v_{[2,1]}\rule{0pt}{12pt}\\
v_{[1,2]} \rule{0pt}{12pt}\\
v_{[2,2]} \rule{0pt}{12pt}\\
v_{[1,3]} \rule{0pt}{12pt}\\
v_{[2,3]} \rule{0pt}{12pt}
\end{array} \right]
=
\left[ \begin{array}{c}
\mbox{vec}({\cal A}_{11}) \\
\mbox{vec}({\cal A}_{21}) \rule{0pt}{12pt}\\
\mbox{vec}({\cal A}_{12}) \rule{0pt}{12pt}\\
\mbox{vec}({\cal A}_{22}) \rule{0pt}{12pt}\\
\mbox{vec}({\cal A}_{13}) \rule{0pt}{12pt}\\
\mbox{vec}({\cal A}_{23}) \rule{0pt}{12pt}
\end{array} \right]
\qquad 
\mbox{vol}_{\mathbf{\smallM}}(\mathbf{i}) \;=\;
\left\{
\begin{array}{ll}
{m^{(1)}_{1}m^{(2)}_{1}} & \mbox{if $\mathbf{i} = [1,1]$}\\
{m^{(1)}_{2}m^{(2)}_{1}}\rule{0pt}{12pt} & \mbox{if $\mathbf{i} = [2,1]$}\\
{m^{(1)}_{1}m^{(2)}_{2}}\rule{0pt}{12pt} & \mbox{if $\mathbf{i} = [1,2]$}\\
{m^{(1)}_{2}m^{(2)}_{2}}\rule{0pt}{12pt} & \mbox{if $\mathbf{i} = [2,2]$} \\
{m^{(1)}_{1}m^{(2)}_{3}} \rule{0pt}{12pt} & \mbox{if $\mathbf{i} = [1,3]$}\\
{m^{(1)}_{2}m^{(2)}_{3}}& \rule{0pt}{12pt} \mbox{if $\mathbf{i} = [2,3]$}
\end{array} \right. .
\]

As we mentioned in the introduction, our goal is to permute the rows and columns of the unfolding
${\cal A}_{\mathbf{r}\times \mathbf{c}}$ so that its blocks are unfoldings of $\cal A$'s blocks. To be more precise, if ${\cal A} = ({\cal A}_{\mathbf{i}})$ is a block tensor
our goal is to determine permutation matrices $P_{\mathbf{\smallR}}$ and 
$P_{\mathbf{\smallC}}$ so that 
\begin{equation}
{\cal A}_{\mathbf{\smallR}\times \mathbf{\smallC}} \;=\; P_{\mathbf{\smallR}} {\cal A}_{\mathbf{r}\times \mathbf{c}} P_{\mathbf{\smallC}}^{T}
\end{equation}
is a block matrix whose blocks are the matrices 
$({\cal A}_{\mathbf{k}})_{\mathbf{r}\times\mathbf{c}}$. 
It turns out that the permutations $ P_{\mathbf{\smallR}}$ and $ P_{\mathbf{\smallC}}$ map ``vec-of-a-tensor'' to ``$\mbox{vec}_{\mathbf{\smallM}}$-of-a-tensor.'' This is not surprising since
the rows and columns of ${\cal A}_{\mathbf{r}\times\mathbf{c}}$  are vec's of reduced order block tensors, see (2.17)-(2.20).

\medskip

\begin{theorem}
 Suppose $\mathbf{M} = \{\mathbf{m}^{(1)},\ldots,\mathbf{m}^{(d)}\}$ is a blocking of 
\inTens{A}{}{n_{1}}{n_{d}} with
\[
\mathbf{m}^{(k)} = [ m^{(1)}_{1} ,\ldots,m^{(k)}_{b_{k}}]\qquad k=1,\ldots,d.
\]
For $k=1,
\ldots,d$  set
\begin{eqnarray*}
N_{k} &=& n_{1} \cdots n_{k},  \\
\mathbf{M}_{k} &=& \{\mathbf{m}^{(1)},\ldots,\mathbf{m}^{(k)}\}, \rule{0pt}{14pt} 
\end{eqnarray*}
and define
\begin{equation}
Q_{k} \;=\; \left\{ 
\begin{array}{ll}
I_{\smallN_{d}} & \mbox{\rm if $k=1$}\\
I_{\smallN_{d}/\smallN_{k}} \T \Gamma^{(k)} & \mbox{\rm if $1<k\leq d$}  \rule{0pt}{14pt}
\end{array} \right.
\end{equation}
where $N_d/N_k = n_{k+1}n_{k+2}\cdots n_d$,
\begin{equation}
\Gamma^{(k)} \;=\; \mbox{\rm diag}(\Gamma^{(k)}_{1},\ldots,\Gamma^{(k)}_{b_{k}})
\end{equation}
and
\begin{equation}
\Gamma^{(k)}_{{j}} \;=\;
\mbox{\rm diag}(\ldots,
\Pi_{\mbox{\rm \footnotesize vol}_{\mathbf{\smallM}_{k-1}}(\mathbf{i}),m^{(k)}_{j}},\ldots )\cdot
\Pi_{m^{(k)}_{j},\smallN_{k-1}}
\qquad  \mathbf{1} \leq \mathbf{i} \leq \mathbf{b}(1:k-1).
\end{equation}
The permutation matrix $P_{\mathbf{\smallM}}$ defined by
\[
P_{\mathbf{\smallM}} \;=\; Q_{d} \cdots Q_{2}Q_{1}
\]
has the property that
\[
\mbox{\rm vec}_{\mathbf{\smallM}}({\cal A}) \;=\; P_{\mathbf{\smallM}} \mbox{\rm vec}({\cal A}).
\]
\end{theorem}
\begin{proof}
Since both  $\mbox{vec}(\cdot)$ and $\mbox{vec}_{\mathbf{\smallM}}(\cdot )$ are linear operators and 
any tensor is the sum of rank-1 tensors, it suffices to prove the theorem for the case
\[
{\cal A} \;=\; a^{(1)} \circ \cdots \circ a^{(d)}
\]
where each \inv{a^{(k)}}{n_{k}} is blocked as follows:
\[
a^{(k)} \;=\;
\left[ 
\begin{array}{c}
a^{(k)}_{1} \\ \vdots \rule[-7pt]{0pt}{17pt} \\ a^{(k)}_{b_{k}} \end{array} \right]\!\!
\begin{array}{l} \left. \right\}m^{(k)}_{1} \\ \rule[-7pt]{0pt}{17pt} \\ \left. \right\}m^{(k)}_{b_{k}}
\end{array}.
\]
We proceed by induction noting that
the theorem is true if $d = 1$ because in that case,
$\mbox{vec}_{\mathbf{\smallM}}({\cal A}) = \mbox{vec}({\cal A})$.
Assume that the theorem holds for block tensors with order $d-1$ or less with $d>1$.
Define
\begin{eqnarray*}
\widehat{\cal A} &=& a^{(1)} \circ \cdots \circ a^{(d-1)}\\
\mathbf{\widehat{M}} &=& \mathbf{M}_{d-1} \rule{0pt}{13pt} \\
\mathbf{\widehat{b}} &=& \mathbf{b}(1:d-1) \rule{0pt}{13pt}.
\end{eqnarray*}
and observe that $\mathbf{\widehat{M}}$
is a blocking for  $\widehat{\cal A}$, an order-$(d-1)$ tensor.
It follows  
by induction that
\begin{equation}
\mbox{vec}_{\mathbf{\widehat{\smallM}}}(\widehat{\cal A}\:) \;=\; 
P_{\mathbf{\widehat{\smallM}}}\,\mbox{vec}(\widehat{\cal A}\:).
\end{equation}
From the definition of $\mbox{vec}_{\mathbf{\smallM}}(\cdot )$ in (3.7), we have
\begin{equation}
\mbox{vec}_{\mathbf{\widehat{\smallM}}}(\widehat{\cal A}\:) \;=\;
\left[ \begin{array}{c} v_{\mathbf{1}} \\ \vdots \\ v_{\widehat{\mathbf{b}}} \end{array} \right]
\qquad
v_{{\mathbf{i}}} = a^{(d-1)}_{i_{d-1}} \T \cdots \T a^{(1)}_{i_{1}}
\end{equation}
for all $\mathbf{i}$ that satisfy $\mathbf{1} \leq \mathbf{i} \leq \mathbf{\widehat{b}}$.
Equation (2.21) says that
\[
\mbox{vec}({\cal A}) \;=\; a^{(d)} \T (a^{(d-1)} \T \cdots \T a^{(1)}) \;=\;
a^{(d)} \T \mbox{vec}(\widehat{\cal A}\:),
\]
and so
\begin{equation}
(I_{n_{d}} \T P_{\mathbf{\widehat{\smallM}}})\mbox{vec}({\cal A})
\;=\; a^{(d)} \T v \;=\;
\left[ \begin{array}{c} a^{(d)}_{1} \\ \vdots \\ a^{(d)}_{b_{d}} \end{array} \right]
\T v \;=\;
\left[ \begin{array}{c} a^{(d)}_{1}\T v \\ \vdots \\ a^{(d)}_{b_{d}}\T v \end{array} \right].
\end{equation}
Using (2.8) we have for $j=1,\ldots,b_{d}$ that
\[
\Gamma^{(d)}_{j}\left( a^{(d)}_{j} \T v \right) \;=\; 
\left[ \begin{array}{c}
a^{(d)}_{j} \T v_{\mathbf{1}}   \\
\vdots \\
a^{(d)}_{j} \T v_{\widehat{\mathbf{b}}} 
\end{array}
\right]
\]
where
\[ 
\Gamma^{(d)}_{j} \;=\; \mbox{diag}\left(\Pi_{\mbox{\scriptsize vol}_{\widehat{\mathbf{M}}}(\mathbf{1}),m^{(d)}_{j}},\ldots,
\Pi_{\mbox{\scriptsize vol}_{\widehat{\mathbf{M}}}(\mathbf{b}(1:d-1)),m^{(d)}_{j}}\right) \cdot \Pi_{m^{(d)}_{j},N/n_{d}}.
\]
Thus, if $\Gamma^{(d)} = \mbox{diag}(\Gamma^{(d)}_{1},\ldots,\Gamma^{(d)}_{b_{d}} )$,
then
\begin{equation}
\Gamma^{(d)} 
\left[ \begin{array}{c} a^{(d)}_{1}\T v \rule[-10pt]{0pt}{15pt} \\ \hline  \vdots \\ \hline 
\rule[-0pt]{0pt}{15pt} a^{(d)}_{b_{d}}\T v \end{array} \right]
\;=\;
\left[ \begin{array}{c} 
a^{(d)}_{1} \T v_{\mathbf{1}}   \\
\vdots \\
a^{(d)}_{1} \T v_{\widehat{\mathbf{b}}}  \rule[-10pt]{0pt}{15pt}
 \\ \hline \vdots \\ \hline 
a^{(d)}_{b_{d}} \T v_{\mathbf{1}}  \rule[-0pt]{0pt}{15pt} \\
\vdots \\
a^{(d)}_{b_{d}} \T v_{\widehat{\mathbf{b}}}  \end{array} \right]
\;=\;\mbox{vec}_{\mathbf{M}}({\cal A}).
\end{equation}
Combining this equation with (3.15)  we have
\[
\Gamma^{(d)}(I_{n_{d}} \T P_{\widehat{\mathbf{\smallM}}}) \mbox{vec}({\cal A}) \;=\; \mbox{vec}_{\mathbf{\smallM}}({\cal A})
\]
and so $P_{\smallM} = \Gamma^{(d)}(I_{n_{d}} \T P_{\widehat{\mathbf{\smallM}}})$.
But by induction
\[
P_{\widehat{\mathbf{\smallM}}} \;=\; \widehat{Q}_{d-1} \cdots \widehat{Q}_{2}\: \widehat{Q}_{1}
\]
where 
\[
\widehat{Q}_{k} \;=\; 
\left\{ \begin{array}{ll}
I_{\smallN_{d-1}} & \mbox{if $k=1$} \\
I_{\smallN_{d-1}/\smallN_{k}} \T \Gamma^{(k)} & \mbox{if $1<k\leq d-1$}
\rule{0pt}{15pt}
\end{array}\right. .
\]
It follows that
\begin{eqnarray*}
P_{\mathbf{\smallM}} &=& \Gamma^{(d)}(I_{n_{d}} \T P_{\widehat{\mathbf{\smallM}}}) \\
 &=& \Gamma^{(d)}(I_{n_{d}} \T \widehat{Q}_{d-1}) \cdots (I_{n_{d}} \T \widehat{Q}_{2}) (I_{n_{d}} \T \widehat{Q}_{1}) \rule{0pt}{13pt} \\
&=& (I_{\smallN_{d}/\smallN_{d}} \T \Gamma^{(d)})(I_{\smallN_{d}/\smallN_{d-1}} \T \Gamma^{(d-1)})
\cdots  (I_{\smallN_{d}/\smallN_{2}} \T \Gamma^{(2)})(I_{\smallN_{d}}) \rule{0pt}{13pt}\\
&=& Q_{d}\: Q_{d-1} \cdots Q_{2} \: Q_{1} \rule{0pt}{13pt}
\end{eqnarray*}
completing the proof.
\end{proof}

\medskip

\noindent
The permutation $P_{\mathbf{\smallM}}$ has a particularly simple form if the blocking is
uniform in each dimension.

\medskip

\begin{corollary}
Suppose $\mathbf{M}$ is defined by {\rm (3.1)-(3.5)}. If 
\begin{eqnarray*}
m^{(k)}_{1} &=& \cdots = m^{(k)}_{b_{k}} \;=\; \mu_{k} \\
N_{k} &=& n_{1} \cdots n_{k} \rule{0pt}{12pt} \\
B_{k} &=& b_{1} \cdots b_{k} \rule{0pt}{12pt}\\
D_{k} &=& \mu_{1} \cdots \mu_{k} \rule{0pt}{12pt}
\end{eqnarray*}
for $k=1,\ldots,d$, then $P_{\smallM} = Q_{d} \cdots Q_{2}Q_{1}$ where
\[
Q_{k} \;=\; \left\{
\begin{array}{ll}
I_{\smallN_{d}} & \mbox{\rm if $k=1$}\\
I_{b_{k}\smallN_{d}/\smallN_{k}} \T \Pi_{\mu_{k},\smallB_{k-1}} \T I_{\smallD_{k-1}}
&\mbox{\rm if $1<k \leq d$} \rule{0pt}{14pt}
\end{array}\right.
\]
\end{corollary}
\begin{proof}
Observe that $\mbox{vol}_{\mathbf{\smallM}_{k-1}}(\mathbf{i}) \;=\; \mu_{1}\cdots \mu_{k-1}$. It follows from the definition of $\Gamma^{(k)}_{j}$
in (3.12)
that
\[
\Gamma^{(k)}_{j} \;=\;  \left(I_{\smallB_{k-1}} \T \Pi_{\smallD_{k-1},\mu_{k}} \right)\Pi_{\mu_{k},\smallN_{k-1}} .
\]
Using the well-known Kronecker product identity 
\[
\left( I_{s} \T \Pi_{r,q} \right)\Pi_{q,rs} \;=\; \Pi_{q,s} \T I_{r},
\]
 it follows that
\[
\Gamma^{(k)}_{j}
\;=\;\Pi_{\mu_{k},\smallB_{k-1}} \T I_{\smallD_{k-1}}.
\]
See \cite{cvFFT}. From (3.10) we have
\[
\Gamma^{(k)} \;=\; I_{b_{k}} \T \Pi_{\mu_{k},\smallB_{k-1}} \T I_{\smallD_{k-1}}
\]
and so
\[
 Q_{k} \;=\; I_{\smallN_{d}/\smallN_{k}} \T \Gamma^{(k)} \;=\;I_{\smallN_{d}b_{k}/\smallN_{k}} \T \Pi_{\mu_{k},\smallB_{k-1}} \T I_{\smallD_{k-1}}.
\]
This completes the proof.
\end{proof}

\medskip

It is interesting to note that the transition from $\mbox{vec}({\cal A})$
to $\mbox{vec}_{\mathbf{\smallM}}({\cal A})$ via the sequence 
\[
Q_{2}\!\cdot\! \mbox{vec}({\cal A})\; \rightarrow \;Q_{3}\!\cdot \!(Q_{2}\!\cdot\!\mbox{vec}({\cal A}))
\; \rightarrow \cdots \rightarrow \; Q_{d}\!\cdot\!(Q_{d-1}\cdots Q_{2}\cdot \mbox{vec}({\cal A}))
\]
is actually a sequence of transpositions. To illustrate, assume
\inv{\cal A}{n_{1}\times n_{2} \times n_{3} \times n_{4}} 
and define the order-8 tensor ${\cal A}^{(1)} $ by
\[
{\cal A}(i_{1},i_{2},i_{3},i_{4}) \;=\;
{\cal A}^{(1)}(\delta_{1},\beta_{1},\delta_{2},\beta_{2},\delta_{3},\beta_{3},\delta_{4},\beta_{4})
\] 
where $\mathbf{1} \leq \mathbf{i} \leq \mathbf{n}$ and the $\delta_{k}$ and $\beta_{k}$ are uniquely defined by
\[
i_{k} = \delta_{k} + (\beta_{k}-1)b_{k}\qquad 1 \leq \delta_{k} \leq  \mu_{k}.
\] 
This says that \inv{{\cal A}^{(1)}}{\mu_{1}\times b_{1} \times \mu_{2}\times b_{2} \times\mu_{3}\times b_{3} \times\mu_{4}\times b_{4}}.
In the $d=4$ case, the $Q$-matrices in Corollary 3.2 are given by
\begin{eqnarray*}
Q_{2} &=& I_{b_{2}n_{3}n_{4}} \T \Pi_{\mu_{2},b_{1}} \T I _{\mu_{1}}\\
Q_{3} &=& I_{b_{3}n_{4}}      \T \Pi_{\mu_{3},b_{1}b_{2}} \T I _{\mu_{1}\mu_{2}} \rule{0pt}{14pt}\\
Q_{4} &=& I_{b_{4}}           \T \Pi_{\mu_{4},b_{1}b_{2}b_{3}} \T I _{\mu_{1}\mu_{2}\mu_{3}}
\rule{0pt}{14pt}.
\end{eqnarray*}
Note from Lemma 2.1  that these permutations correspond to transpositions.
Indeed, if we
define the tensors ${\cal A}^{(2)}$,  ${\cal A}^{(3)}$  ${\cal A}^{(4)}$ by
\[
\left. 
\begin{array}{r}
{\cal A}^{(2)}(\delta_{1},\delta_{2},\beta_{1},\beta_{2},\delta_{3},\beta_{3},\delta_{4},\beta_{4})\\
{\cal A}^{(3)}(\delta_{1},\delta_{2},\delta_{3},\beta_{1},\beta_{2},\beta_{3},\delta_{4},\beta_{4})\rule{0pt}{16pt}\\
{\cal A}^{(4)}(\delta_{1},\delta_{2},\delta_{3},\delta_{4},\beta_{1},\beta_{2},\beta_{3},\beta_{4}) \rule{0pt}{16pt}
\end{array}
\right\} \quad = \;\;
{\cal A}^{(1)}(\delta_{1},\beta_{1},\delta_{2},\beta_{2},\delta_{3},\beta_{3},\delta_{4},\beta_{4})\\
\]
then it can be shown  via Lemma 2.1 that
\begin{eqnarray*}
\mbox{vec}({\cal A}^{(1)}) & = & Q_{1} \mbox{vec}({\cal A})\;=\; \mbox{vec}({\cal A})\\
\mbox{vec}({\cal A}^{(2)}) & = & Q_{2} \mbox{vec}({\cal A}^{(1)}) \rule{0pt}{12pt}\\
\mbox{vec}({\cal A}^{(3)}) & = & Q_{3} \mbox{vec}({\cal A}^{(2)}) \rule{0pt}{12pt}\\
\mbox{vec}_{\mathbf{\smallM}}({\cal A})\;=\;\mbox{vec}({\cal A}^{(4)}) & = & Q_{4} \mbox{vec}({\cal A}^{(3)}) \rule{0pt}{12pt}.
\end{eqnarray*}
Thus, the order-8 tensor ${\cal A}^{(4)}$ has the property that $\mbox{vec}({\cal A}^{(4)})
= \mbox{vec}_{\mathbf{\smallM}}({\cal A})$. Moreover,  ${\cal A}(\mathbf{i}) = 
{\cal A}_{{\boldbeta}}({\bolddelta})$ showing that entry $\mathbf{i}$ is
entry ${\bolddelta}$ of block ${\boldbeta}$.

\subsection{Block Unfoldings}
We now specify the permutation matrices $P_{\mathbf{\smallR}}$ and $P_{\mathbf{\smallC}}$ in (3.9)
that turn ${\cal A}_{\mathbf{r} \times \mathbf{c}}$ into a block matrix with block entries
that are $\mathbf{r}\times \mathbf{c}$ unfoldings of $\cal A$'s
blocks.

\medskip

\begin{theorem}
Suppose $\mathbf{M} = \{\mathbf{m}^{(1)},\ldots,\mathbf{m}^{(d)}\}$ is a blocking 
of \inTens{A}{}{n_{1}}{n_{d}} with 
\[
\mathbf{m}^{(k)} \:=\: [ m_{1}^{(k)},\ldots,m_{b_{k}}^{(k)}] \qquad k=1,\ldots,d.
\]
Let $e$ be an integer that satisfies
$1 \leq e < d$ and assume that
$\mathbf{p}$ is a permutation of $1:d$. Define
\[
\begin{array}{lclclclclcl}
\mathbf{r} \!&\!=\!&\! \mathbf{p}(1:e) & &
\mathbf{R}  \!&\!=\!&\! \{ \mathbf{m}^{(r_{1})},\ldots,\mathbf{m}^{(r_{e})}\} & &
B_{rows}  \!&\!=\!&\! b_{r_{1}}\cdots b_{r_{e}}\\
\mathbf{c}  \!&\!=\!&\!  \mathbf{p}(e+1:d) & &
\mathbf{C}  \!&\!=\!&\! \{ \mathbf{m}^{(c_{1})},\ldots,\mathbf{m}^{(c_{d-e})}\} & &
B_{cols}  \!&\!=\!&\!  b_{c_{1}}\cdots b_{c_{d-e}}\rule{0pt}{12pt}.
\end{array}
\]
The matrix
\[
{\cal A}_{\mathbf{\smallR}\times \mathbf{\smallC}}\;=\;
P_{\mathbf{\smallR}} {\cal A}_{\mathbf{r} \times \mathbf{c}} P_{\mathbf{\smallC}}^{T}
\]
is a $B_{rows}$-by-$B_{cols}$ block matrix whose block entries are specified by
\begin{equation}
\left( {\cal A}_{\mathbf{\smallR}\times \mathbf{\smallC}} \right)_{\mathbf{k}(\mathbf{r}),\mathbf{k}(\mathbf{c})}
\;=\;
\left( {\cal A}_{\mathbf{k}} \right)_{\mathbf{r} \times \mathbf{c}}
\qquad \mathbf{1}\leq\mathbf{k}\leq\mathbf{b}.
\end{equation}
That is to say, if $\mu = ivec(\mathbf{k}(\mathbf{r}),\mathbf{b}(\mathbf{r}))$ and 
$\tau = ivec(\mathbf{k}(\mathbf{c}),\mathbf{b}(\mathbf{c}))$, then the $(\mu,\tau)$ block of 
${\cal A}_{\mathbf{\smallR}\times \mathbf{\smallC}} $ is the $\mathbf{r}\times \mathbf{c}$
unfolding of the $\mathbf{k}$-th block of $\cal A$.
\end{theorem}

\medskip

\begin{proof}
By linearity there is no loss of generality in assuming that
\[
{\cal A} \;=\; a^{(1)} \circ \cdots \circ a^{(d)}
\]
where each \inv{a^{(k)}}{n_{k}} is blocked as follows:
\[
a^{(k)} \;=\;
\left[ 
\begin{array}{c}
a^{(k)}_{1} \\ \vdots \rule[-7pt]{0pt}{17pt} \\ a^{(k)}_{b_{k}} \end{array} \right]\!\!
\begin{array}{l} \left. \right\}m^{(k)}_{1} \\ \rule[-7pt]{0pt}{17pt} \\ \left. \right\}m^{(k)}_{b_{k}}
\end{array}.
\]
From (2.24) we know that
\[
{\cal A}_{\mathbf{r}\times \mathbf{c}} \;=\; 
\mbox{vec}(\:a^{(r_{1})} \circ \cdots \circ a^{(r_{e})}\:) \cdot
\mbox{vec}(\:a^{(c_{1})} \circ \cdots \circ a^{(c_{d-e})}\:)^{T}.
\]
Since $\mathbf{R}$ is a blocking for $a^{(r_{1})} \circ \cdots \circ a^{(r_{e})}$
and
$\mathbf{C}$ is a blocking for $a^{(c_{1})} \circ \cdots \circ a^{(c_{d-e})}$,
it follows from Theorem 3.1 that
\[
P_{\mathbf{\smallR}}A_{\mathbf{r}\times \mathbf{c}} P_{\mathbf{\smallC}}^{T} \;=\; yz^{T}
\]
where 
$y \;=\; \mbox{vec}_{\mathbf{\smallR}}(a^{(r_{1})} \circ \cdots \circ a^{(r_{e})})$
and
$z \;=\; \mbox{vec}_{\mathbf{\smallC}}(a^{(c_{1})} \circ \cdots \circ a^{(c_{d-e})})$.
These block vectors are specified by

\begin{eqnarray}
y &=& 
\left[ \begin{array}{c} y_{\mathbf{1}} \\ \vdots \\ y_{\mathbf{b}(\mathbf{r})} \end{array}\right]
\qquad y_{\mathbf{i}} \;=\;
\mbox{vec}(a^{(r_{1})}_{i_{1}} \circ \cdots \circ a^{(r_{e})}_{i_{e}})
\qquad \quad \: \mathbf{1} \leq \mathbf{i} \leq  \mathbf{b}(\mathbf{r}) \\
z &=& \left[ \begin{array}{c} z_{\mathbf{1}} \\ \vdots \\ z_{\mathbf{b}(\mathbf{c})} \end{array} \right]
\qquad z_{\mathbf{j}} \;=\;
\mbox{vec}(a^{(c_{1})}_{j_{1}} \circ \cdots \circ a^{(c_{d-e})}_{j_{d-e}})
\qquad \; \mathbf{1} \leq \mathbf{j} \leq \mathbf{b}(\mathbf{c})\rule{0pt}{40pt}
\end{eqnarray}
and so the $(\mathbf{i},\mathbf{j})$-th block of
${\cal A}_{\mathbf{\smallR} \times \mathbf{\smallC}}$ is given by
\begin{equation}
\left( {\cal A}_{\mathbf{\smallR} \times \mathbf{\smallC}}\right)_{\mathbf{i},\mathbf{j}}
\;=\;y_{\mathbf{i}}\:z_{\mathbf{j}}^{T} .
\end{equation}
On the other hand, from (3.17)
\begin{eqnarray*}
({\cal A}_{\mathbf{k}})_{\mathbf{r}\times \mathbf{c}} &=&
\left(
a^{(1)}_{k_{1}} \circ \cdots \circ a^{(d)}_{k_{d}}
\right)_{\mathbf{r}\times \mathbf{c}}\\
&=&
\mbox{vec}\left(a^{(r_{1})}_{k_{r_{1}}} \circ \cdots \circ a^{(r_{e})}_{k_{r_{e}}}\right) \cdot \mbox{vec}\left(a^{(c_{1})}_{k_{c_{1}}} \circ \cdots \circ a^{(c_{d-e})}_{k_{c_{d-e}}}\right)^{T}\rule{0pt}{20pt}.
\end{eqnarray*}
It follows from (3.18)-(3.20) that if $\mathbf{i} = \mathbf{k}(\mathbf{r})$ and
 $\mathbf{j} = \mathbf{k}(\mathbf{c})$, then
\[
({\cal A}_{\mathbf{k}})_{\mathbf{r}\times \mathbf{c}} \;=\; y_{\mathbf{i}} \: z_{\mathbf{j}}^{T}
\;=\; \left( {\cal A}_{\mathbf{\smallR} \times \mathbf{\smallC}}\right)_{\mathbf{i},\mathbf{j}}
\]
which completes the proof.
\end{proof}

\medskip

 To illustrate the theorem, suppose $\cal A$ is 2-by-4-by-3-by-2 block tensor. 
If  $\mathbf{r} = [1\;3]$ and   $\mathbf{c} = [2\;4]$, then
\[
{\cal A}_{\mathbf{\smallR}\times \mathbf{\smallC}} \;=\;
\left[
\begin{array}{cccccccc}
\widetilde{\cal A}_{1111} & 
\widetilde{\cal A}_{1211} & 
\widetilde{\cal A}_{1311} & 
\widetilde{\cal A}_{1411} & 
\widetilde{\cal A}_{1112} & 
\widetilde{\cal A}_{1212} &
\widetilde{\cal A}_{1312} & 
\widetilde{\cal A}_{1412} \\
\widetilde{\cal A}_{2111} & 
\widetilde{\cal A}_{2211} & 
\widetilde{\cal A}_{2311} & 
\widetilde{\cal A}_{2411} & 
\widetilde{\cal A}_{2112} & 
\widetilde{\cal A}_{2212} &
\widetilde{\cal A}_{2312} & 
\widetilde{\cal A}_{2412} \rule{0pt}{13pt}\\
\widetilde{\cal A}_{1121} & 
\widetilde{\cal A}_{1221} & 
\widetilde{\cal A}_{1321} & 
\widetilde{\cal A}_{1421} & 
\widetilde{\cal A}_{1122} & 
\widetilde{\cal A}_{1222} &
\widetilde{\cal A}_{1322} & 
\widetilde{\cal A}_{1422} \rule{0pt}{13pt} \\
\widetilde{\cal A}_{2121} & 
\widetilde{\cal A}_{2221} & 
\widetilde{\cal A}_{2321} & 
\widetilde{\cal A}_{2421} & 
\widetilde{\cal A}_{2122} & 
\widetilde{\cal A}_{2222} &
\widetilde{\cal A}_{2322} & 
\widetilde{\cal A}_{2422} \rule{0pt}{13pt} \\
\widetilde{\cal A}_{1131} & 
\widetilde{\cal A}_{1231} & 
\widetilde{\cal A}_{1331} & 
\widetilde{\cal A}_{1431} & 
\widetilde{\cal A}_{1132} & 
\widetilde{\cal A}_{1232} &
\widetilde{\cal A}_{1332} & 
\widetilde{\cal A}_{1432} \rule{0pt}{13pt}\\
\widetilde{\cal A}_{2131} & 
\widetilde{\cal A}_{2231} & 
\widetilde{\cal A}_{2331} & 
\widetilde{\cal A}_{2431} & 
\widetilde{\cal A}_{2132} & 
\widetilde{\cal A}_{2232} &
\widetilde{\cal A}_{2332} & 
\widetilde{\cal A}_{2432} \rule{0pt}{13pt}
\end{array} \right]
\begin{array}{l}
\scriptstyle(1,1) \rule{0pt}{13pt}\\ \scriptstyle(2,1)\rule{0pt}{13pt} \\ \scriptstyle(1,2) \rule{0pt}{13pt} \\ \scriptstyle(2,2) \rule{0pt}{13pt}\\ \scriptstyle(1,3) \rule{0pt}{13pt} \\ \scriptstyle (2,3) \rule{0pt}{13pt}\end{array}
\]
\vspace*{-.15in}
\[
\;\;\;\;\;\;\begin{array}{llllllll}
\scriptstyle (1,1)\; \,& \;\,\scriptstyle (2,1) \;\,& \;\,\scriptstyle(3,1) \;\,& \;\,\scriptstyle(4,1)\;\, & \;\,
\scriptstyle(1,2) \;\,&\;\, \scriptstyle(2,2)\;\, &\;\, \scriptstyle(3,2)\; \,& \;\,\scriptstyle(4,2) 
\end{array}
\]
where $\widetilde{\cal A}_{\alpha\beta\gamma\delta} \;=\; ({\cal A}_{\alpha\beta\gamma\delta})_{\mathbf{r}\times \mathbf{c}}$. Note the multi-indexing of the block rows and columns.

\subsection{A Special Case}
Returning to the second example in \S2.4, suppose 
\[
{\cal B} \;=\; B^{(1)}\circ \cdots \circ B^{({d}\,)}
\]
where 
\[
\inm{B^{(\ell)}}{q_{\ell}}{n_{\ell}}
\]
for $\ell=1,\ldots,{d}$.
Assume that
$[\mathbf{u}^{(\ell)},\mathbf{v}^{(\ell)}]$ is a blocking for $B^{(\ell)}$ and
note that
\begin{equation}
\mathbf{M} \;=\; \left\{ \mathbf{u}^{(1)},\mathbf{v}^{(1)},\ldots,
\mathbf{u}^{({d}\,)},\mathbf{v}^{({d}\,)} \right\}
\end{equation}
is a blocking for $\cal B$.
Let $B^{(\ell)}_{\mu,\tau}$ denote block $(\mu,\tau)$ of
$B^{(\ell)}$.  If
\[
\mathbf{k} \;=\;[\:i_{1},j_{1},\ldots,i_{{d}},j_{{d}}\:]
\]
then the $\mathbf{k}$-th block of $\cal B$ is given by
\[
{\cal B}_{\mathbf{k}} \;=\; B^{(1)}_{i_{1},j_{1}} \circ \cdots \circ B^{({d}\,)}_{i_{{d}},j_{{d}}}.
\]
If 
\begin{eqnarray}
\mathbf{r} &=& 1:2:2d \nonumber \\
\mathbf{c} &=&  2:2:2d  \rule{0pt}{14pt} \nonumber \\
\mathbf{R} &=&  \left\{ \mathbf{u}^{(1)},\ldots,
\mathbf{u}^{({d}\,)} \right\} \rule{0pt}{14pt}
\\
\mathbf{C} &=&  \left\{\mathbf{v}^{(1)},\ldots,\mathbf{v}^{({d}\,)} \right\},
\rule{0pt}{16pt}
\end{eqnarray}
then by applying (3.17) and (2.25) we see that
\begin{equation}
\left( {\cal B}_{\mathbf{\smallR}\times\mathbf{\smallC}}\right)_{\mathbf{i},\mathbf{j}}
\;=\; \left( B^{(1)}_{i_{1},j_{1}} \circ \cdots \circ B^{({d}\,)}_{i_{{d}},j_{{d}}}
\right)_{\mathbf{r}\times\mathbf{c}}
\;=\;
B^{({d}\,)}_{i_{{d}},j_{{d}}} \T \cdots \T B^{(1)}_{i_{1},j_{1}} .
\end{equation}
Here, the notation $\left( {\cal B}_{\mathbf{\smallR}\times\mathbf{\smallC}}\right)_{\mathbf{i},\mathbf{j}}$ denotes
block $(ivec(\mathbf{i},\mathbf{q}),ivec(\mathbf{j},\mathbf{n}))$.
This result is key to the development of a block-level 
multilinear product which we pursue in \S4.2.

\section{Blocked Contractions}
We next apply our block tensor ``technology'' to the problem of computing a
contraction between two tensors. A multi-index summation notation will be used 
to describe the summations. If $\mathbf{n}$ is a length-$d$
index vector, then
\[
\sum_{\mathbf{i}=\mathbf{1}}^{\mathbf{n}} \;\equiv\; 
\sum_{i_{1}=1}^{n_{1}} \cdots \sum_{i_{d}=1}^{n_{d}}.
\]
\subsection{The General Case}
It is instructive to work through a small, motivating example
before we present the main results. 
Suppose we are given \inTens{F}{}{\alpha_{1}}{\alpha_{4}}
and  \inTens{G}{}{\beta_{1}}{\beta_{5}} and wish to compute the
order-5 tensor \inv{\cal H}{\alpha_{3}\times\alpha_{4}\times\beta_{3}\times\beta_{4}\times\beta_{5}} defined by
\begin{equation}
{\cal H}(i_{1},i_{2},j_{1},j_{2},j_{3}) \;=\; \sum_{k_{1}=1}^{\alpha_{3}} \sum_{k_{2}=1}^{\alpha_{4}}
{\cal F}(i_{1},i_{2},k_{1},k_{2}) \cdot
{\cal G}(k_{1},k_{2},j_{1},j_{2},j_{3}) .
\end{equation}
Of course, for this to make sense, we must have $\alpha_{3} = \beta_{1}$ and $\alpha_{4} = \beta_{2}$.
It is well known that a tensor contraction such as this can be ``reshaped'' into a single matrix-matrix multiplication. To see this we  rewrite  (4.1) using multi-index notation,
\begin{equation}
{\cal H}(\mathbf{i},\mathbf{j}) \;=\; 
\sum_{\mathbf{k}=\mathbf{1}}^{{\scriptsize\boldalpha}(3:4)} 
{\cal F}(\mathbf{i},\mathbf{k}) \cdot
{\cal G}(\mathbf{k},\mathbf{j}) .
\end{equation}
Define the index vectors
\begin{equation*}
\boldsymbol{r}\: =\: [ \:{1} \; {2}\:] \qquad
\boldsymbol{\lambda}\: =\: [ \:{3} \; {4}\:] \qquad
 \boldsymbol{\psi}\: = \:[\: {1} \; {2}\:] \qquad
\mathbf{c}\: =\: [ \:{3} \; {4} \; {5}\:] 
\end{equation*}
and note that $\mathbf{1}\leq \mathbf{i} \leq \boldsymbol{\alpha}(\mathbf{r})$ and $\mathbf{1}\leq \mathbf{j} \leq \boldsymbol{\beta}(\mathbf{c})$ in (4.2). Recall from (2.17)-(2.20) that the rows and columns of a tensor unfolding are vecs of reduced-order subtensors. In particular
\begin{eqnarray*}
{\cal F}_{\mathbf{r}\times {{\boldsymbol{\lambda}}}}(\mathbf{i},:) &=& \mbox{vec}({\cal F}^{(\mathbf{i})})^{T}\\
{\cal G}_{{\boldsymbol{\psi}}\times\mathbf{c}}(:,\mathbf{j}) &=& \mbox{vec}({\cal G}^{(\mathbf{j})})
\rule{0pt}{13pt}
\end{eqnarray*}
where \inv{{\cal F}^{(\mathbf{i})}}{\alpha_{3}\times\alpha_{4}} and
\inv{{\cal G}^{(\mathbf{j})}}{\beta_{1}\times\beta_{2}} are defined by
\[
\begin{array}{lclclcl}
{\cal F}^{(\mathbf{i})}(k_{1},k_{2}) \!&\!=\!&\! {\cal F}(i_{1},i_{2},k_{1},k_{2}) & & \mathbf{i} \!&\!=\!&\![\:i_{1} \: i_{2}\:]\\
{\cal G}^{(\mathbf{j})}(k_{1},k_{2}) \!&\!=\!&\! {\cal G}(k_{1},k_{2},j_{1},j_{2},j_{3}) \rule{0pt}{13pt}& & \mathbf{j} \!&\!=\!&\![\:j_{1} \: j_{2}\:j_{3}\:].
\end{array}
\]
It follows from (4.2) that
\[
{\cal H}(\mathbf{i},\mathbf{j}) \;=\; 
\sum_{k_{1}=1}^{\alpha_{1}} \sum_{k_{2}=1}^{\alpha_{2}}
{\cal F}^{(\mathbf{i})}(k_{1},k_{2})\cdot {\cal G}^{(\mathbf{j})}(k_{1},k_{2}) \;=\;
{\cal F}_{  \mathbf{r}\times \boldsymbol{\lambda}}(\mathbf{i},:)\!\cdot
{\cal G}_{\boldsymbol{\psi}\times\mathbf{c}}(:,\mathbf{j})
\]
and thus 
\[
{\cal H}_{[1\:2]\times[3\:4\:5]} \;=\;
{\cal F}_{\mathbf{r}\times \boldsymbol{\lambda}}\!\! \cdot 
{\cal G}_{\boldsymbol{\psi}\times\mathbf{c}}.
\]
In this example, the summation is over the last two modes of $\cal F$ and the first two modes of
$\cal G$. These are convenient locations for the summation indices because the contraction
$\cal H$ is then easily seen to be ``isomorphic'' to a matrix-matrix product of simple tensor unfoldings.

If the summation modes are arbitrarily positioned, then they can be 
moved to these friendly locations through transposition. This result is widely known and
exploited, e.g., \cite{TCE, koldaToolbox}. Nevertheless,  in keeping with the spirit of
this paper we think that it is useful to include a 
formal verification of this important maneuver.

\medskip

\begin{theorem}
Suppose \inTens{F}{}{\alpha_{1}}{\alpha_{f+\ell}}, \inTens{G}{}{\beta_{1}}{\beta_{g+\ell}},
and that $\mathbf{p}$ and $\mathbf{q}$ are permutations of
$1:f\!+\!\ell$ and $1:g\!+\!\ell$ respectively. Define
\begin{equation*}
\begin{array}{lclclcl}
\mathbf{r} \!&\!=\!&\! \mathbf{p}(1:f) &\quad &
\boldsymbol{\lambda} \!&\!=\!&\! \mathbf{p}((f\!+\!1):(f\!+\!\ell) )
    \\
\boldsymbol{\psi} \!&\!=\!&\! \mathbf{q}(1:\,\ell)  &  &  \mathbf{c}\!&\!=\!&\! \mathbf{q}((\,\ell\!+\!1):(\ell\!+\!g))  \rule{0pt}{14pt}
\end{array}
\end{equation*}
and assume $\mathbf{\boldalpha}(\mathbf{\boldlambda}) = \mathbf{\boldbeta}({\boldpsi})$.
If \inv{\cal H}{\alpha_{r_{1}}\times\cdots \times\alpha_{r_{f}}\times
\beta_{c_{1}}\times\cdots \times\beta_{c_{g}}}
is defined by
\begin{equation}
{\cal H}(\mathbf{i},\mathbf{j}) \;=\; \sum_{\mathbf{k}=\mathbf{1}}^{{\scriptsize\boldalpha}({\boldlambda})}
{\cal F}^{<\mathbf{p}>}(\mathbf{i},\mathbf{k})\:{\cal G}^{<\mathbf{q}>}(\mathbf{k},\mathbf{j})
\qquad \quad \mathbf{1}\leq \mathbf{i} \leq {\boldalpha}(\mathbf{r}),\;\;\;
\mathbf{1}\leq \mathbf{j} \leq {\boldbeta}(\mathbf{c}),
\end{equation}
then
\begin{equation}
{\cal H}_{[1:f]\times [f+1:f+g]} \;=\;{\cal F}_{\boldsymbol{\mathbf{r}\times \lambda}} \! \! \cdot
{\cal G}_{\boldsymbol{\psi}\times\mathbf{c}}.
\end{equation}
\end{theorem} \medskip
\begin{proof}
The assumption ${\boldalpha}({\boldlambda}) = {\boldbeta}({\boldpsi})$ ensures that the summations in (4.3) are well defined.
Using (2.17)-(2.20) we have
\begin{eqnarray*}
{\cal F}_{\mathbf{r}\times\boldsymbol{\lambda}}(\mathbf{i},:) &=& \mbox{vec}({\cal F}^{(\mathbf{i})})^{T}\\
{\cal G}_{\boldsymbol{\psi}\times\mathbf{c}}(:,\mathbf{j}) &=& \mbox{vec}({\cal G}^{(\mathbf{j})})
\rule{0pt}{13pt}
\end{eqnarray*}
where \inv{{\cal F}^{(\mathbf{i})}}{\alpha_{\lambda_{1}}\times\cdots \times \alpha_{\lambda_{\ell}}} and
\inv{{\cal G}^{(\mathbf{j})}}{\beta_{\psi_{1}}\times\cdots \times \beta_{\psi_{\ell}}} are defined by
\begin{eqnarray*}
{\cal F}^{(\mathbf{i})}(\mathbf{k}) &=& {\cal F}^{<\mathbf{p}>}(i_{1},\ldots,i_{f},k_{1},\ldots,k_{\ell})\\
{\cal G}^{(\mathbf{j})}(\mathbf{k}) &=& {\cal G}^{<\mathbf{q}>}(k_{1},\ldots,k_{\ell},j_{1},\ldots,j_{g}) \rule{0pt}{13pt}.
\end{eqnarray*}
It follows that for all $\mathbf{i}$ and $\mathbf{j}$ that satisfy
$\mathbf{1}\leq \mathbf{i} \leq {\boldalpha}(\mathbf{r})$ and
$\mathbf{1}\leq \mathbf{j} \leq {\boldbeta}(\mathbf{c})$ we have
\begin{eqnarray*}
{\cal H}(\mathbf{i},\mathbf{j}) &=& \sum_{\mathbf{k}=\mathbf{1}}^{\boldsymbol{\alpha}(\boldsymbol{\lambda})}
{\cal F}^{<\mathbf{p}>}(\mathbf{i},\mathbf{k})\cdot {\cal G}^{<\mathbf{q}>}(\mathbf{k},\mathbf{j})\\
& = &  \sum_{\mathbf{k}=\mathbf{1}}^{{\boldsymbol{\alpha}}(\boldsymbol{\lambda})} 
{\cal F}^{(\mathbf{i})}(\mathbf{k})\cdot {\cal G}^{(\mathbf{j})}(\mathbf{k}) 
\;=\;
{\cal F}_{\boldsymbol{\mathbf{r}\times \lambda}}(\mathbf{i},:)\!\cdot
{\cal G}_{\boldsymbol{\psi}\times\mathbf{c}}(:,\mathbf{j})
\end{eqnarray*}
which, using (2.14)-(2.18), implies (4.4).
\end{proof}

\bigskip

It is instructive to illustrate what the theorem ``says'' when $\mathbf{c} = \emptyset$.
Suppose ${\cal F} \in \mathbb{R}^{\alpha_1\times \cdots \times \alpha_5}$ and  ${\cal G} \in \mathbb{R}^{\beta_1\times \beta_2}$ with $\alpha_2=\beta_2$, $\alpha_3 = \beta_1$.
If the tensor ${\cal H} \in \mathbb{R}^{\alpha_5\times \alpha_1 \times \alpha_4}$ is defined by the contraction
\[
{\cal H}(i_1,i_2,i_3) \;=\; \sum_{\mathbf{k}=\mathbf{1}}^{\boldsymbol{\alpha}(2:3)} {\cal F}(i_2,k_1,k_2,i_3,i_1) {\cal G}(k_2,k_1),
\]
then in the notation of the theorem we have $f=3, \ell = 2, g=0$, 
$\mathbf{p} \:=\:[ \:5 \; 1\; 4\; 2\; 3\:]$, and $\mathbf{q} \:= \: [\: 2 \; 1 \:].$
It follows that
$\mathbf{r} \: = \: [ \: 5 \; 1 \; 4\:]$,
$\mathbf{c} \: = \: \emptyset$,
$\boldsymbol{\lambda} \: = \: [ \: 2 \; 3 \:]$, and
$\boldsymbol{\psi} \: =\: [\: 2 \; 1 \:]$. Thus, we may conclude from
(4.4) that
\[
{\cal H}_{[1:3]\times \emptyset} \;=\; \mbox{vec}({\cal H}) \;=\; {\cal F}_{[5\; 1\; 4]\times[2\; 3]} \cdot {\cal G}_{[2 \; 1]\times \emptyset} \\
\;=\; {\cal F}_{[5\; 1\; 4]\times[2\; 3]} \cdot \mbox{vec}( {\cal G}^T),
\]
a matrix-vector product.
\medskip

If the tensors $\cal F$ and $\cal G$ are ``blocked conformally'', then (4.3) can be reformulated as a product of two block matrices.

\medskip

\begin{corollary}
Assume that the notation and conditions of  Theorem 4.1 hold. Let
\begin{equation}
\mathbf{S} \;=\; \{\mathbf{s}^{(1)},\ldots,\mathbf{s}^{(f+\ell)}\} 
\end{equation}
be a {blocking} for   $\cal F$ and set
\[
\mathbf{R} \;=\; \{\mathbf{s}^{(r_{1})},\ldots,\mathbf{s}^{(r_{f})}\}
\qquad
\mathbf{\Lambda} \;=\;\{\mathbf{s}^{(\lambda_{1})},\ldots,\mathbf{s}^{(\lambda_{\ell})}\}.
\]
Likewise, let
\begin{equation}
\mathbf{T} \;=\; \{\mathbf{t}^{(1)},\ldots,\mathbf{t}^{(g+\ell)}\}
\end{equation}
 be a blocking for $\cal G$ and set
\[
\mathbf{\Psi} \;=\; \{\mathbf{t}^{(\psi_{1})},\ldots,\mathbf{t}^{(\psi_{\ell})}\} 
\qquad 
\mathbf{C} \;=\; \{\mathbf{t}^{(c_{1})},\ldots,\mathbf{t}^{(c_{g})}\}. \rule{0pt}{13pt}
\]
If
\vspace*{-.1in}
\begin{equation}
\mathbf{s}^{(\lambda_{k})} =  \mathbf{t}^{(\psi_{k})} \qquad k=1,\ldots,\ell
\end{equation}
then with respect to the tensor $\cal H$, 
$\mathbf{R}$ is a blocking for modes $1$ through $f$,
$\mathbf{C}$ is a blocking for modes $f+1$ through $f+g$, and 
\begin{equation}
{\cal H}_{{\mathbf{\smallR}}\times{\mathbf{\smallC}}} \;=\;
{\cal F}_{\mathbf{\smallR}\times\mathbf{\smallLambda}}\cdot
{\cal G}_{\mathbf{\smallPsi}\times\mathbf{\smallC}}.
\end{equation}
\end{corollary}
\begin{proof}
From Theorem 3.3 we have
\[
{\cal F}_{\mathbf{\smallR}\times \mathbf{\smallLambda}} \:=\:
P_{\mathbf{\smallR}}{\cal F}_{\mathbf{r}\times\boldsymbol{\lambda}}P_{\mathbf{\smallLambda}}^{T}\qquad \qquad
{\cal G}_{\mathbf{\smallPsi}\times \mathbf{\smallC}} \:=\:
P_{\mathbf{\smallPsi}}{\cal G}_{\boldsymbol{\psi}\times\mathbf{c}}P_{\mathbf{\smallC}}^{T}. 
\]
Since
$\{\mathbf{s}^{(r_{1})},\ldots,\mathbf{s}^{(r_{f})},\mathbf{t}^{(c_{1})},\ldots,\mathbf{t}^{(c_{g})}\}$ is a blocking for $\cal H$ we also have
\[
{\cal H}_{{\mathbf{\smallR}}\times {\mathbf{\smallC}}} \;=\; P_{\mathbf{\smallR}}\cdot {\cal H}_{[1\,:\,f]\times[f+1\,:\,f+g]}\,\cdot P_{\mathbf{\smallC}}^{T} .
\]
The conformability condition (4.7) implies
$P_{\mathbf{\smallLambda}} = P_{\mathbf{\smallPsi}}$ and so it follows from  (4.4) that
\begin{eqnarray*}
{\cal H}_{{\mathbf{\smallR}}\times{\mathbf{\smallC}}} &=&
P_{\mathbf{\smallR}}({\cal F}_{\mathbf{r}\times\boldsymbol{\lambda}}\cdot
{\cal G}_{\boldsymbol{\psi}\times\mathbf{c}})P_{\mathbf{\smallC}}^{T} \rule{0pt}{12pt}\\
&=& (P_{\mathbf{\smallR}}{\cal F}_{\mathbf{r}\times\boldsymbol{\lambda}}P_{\mathbf{\smallLambda}}^{T} )
(P_{\mathbf{\smallPsi}}{\cal G}_{\boldsymbol{\psi}\times\mathbf{c}}P_{\mathbf{\smallC}}^{T} ) \rule{0pt}{12pt}
\;=\; {\cal F}_{\mathbf{\smallR}\times\mathbf{\smallLambda}}\cdot
{\cal G}_{\mathbf{\smallPsi}\times\mathbf{\smallC}} \rule{0pt}{12pt}
\end{eqnarray*}
completing the proof.
\end{proof}

\medskip

\noindent
Thus, the tensor $\cal H$ in (4.3) can be computed  as either a matrix product (4.4) or as  a block matrix product (4.8). For the latter case,
we develop recipes for the blocks
of ${\cal H}_{{\mathbf{\smallR}}\times{\mathbf{\smallC}}}$. Let $b^{(\mathbf{\smallS})}_{j}$ 
be the length of the blocking vector  $\mathbf{s}^{(j)}$ in (4.5)  and
let $b^{(\mathbf{\smallT})}_{j}$ be the length of the blocking vector $\mathbf{t}^{(j)}$
in (4.6). Note that if
\[
\begin{array}{lclclcl}
b_{rows}^{({\cal F})} &\!\!=\!\!& b^{(\mathbf{\smallS})}_{r_{1}}\cdots b^{(\mathbf{\smallS})}_{r_{f}} &\quad &
b_{cols}^{({\cal F})} &\!\!=\!\!& b^{(\mathbf{\smallS})}_{\lambda_{1}}\cdots b^{(\mathbf{\smallS})}_{\lambda_{\ell}} \\
b_{rows}^{({\cal G})} &\!\!=\!\!& b^{(\mathbf{\smallT})}_{\psi_{1}}\cdots b^{(\mathbf{\smallT})}_{\psi_{\ell}} &\quad &
b_{cols}^{({\cal G})} &\!\!=\!\!& b^{(\mathbf{\smallT})}_{c_{1}}\cdots b^{(\mathbf{\smallT})}_{c_{g}}\rule{0pt}{16pt}
\end{array}
\]
then (4.7) implies $b_{cols}^{({\cal F})}  = b_{rows}^{({\cal G})}$
and we observe that
\[
\left\{
\begin{array}{l}
{\cal F}_{\mathbf{\smallR}\times\mathbf{\smallLambda}} \\
{\cal G}_{\mathbf{\smallLambda}\times\mathbf{\smallC}} \rule{0pt}{16pt}\\
{\cal H}_{{\mathbf{\smallR}}\times{\mathbf{\smallC}}}  \rule{0pt}{16pt}
\end{array}
\right\} \quad \mbox{is a } \quad \left\{
\begin{array}{l}
b_{rows}^{({\cal F})}\mbox{-by-}b_{cols}^{({\cal F})}\\
b_{rows}^{({\cal G})}\mbox{-by-}b_{cols}^{({\cal G})} \rule{0pt}{16pt}\\
b_{rows}^{({\cal F})}\mbox{-by-}b_{cols}^{({\cal G})} \rule{0pt}{16pt}
\end{array}
\right\}
\quad \mbox{block matrix.}
\]
If $\mathbf{1}\leq \boldsymbol{\mu} \leq \mathbf{b}^{(\mathbf{\smallS})}(\mathbf{r})$
and
$\mathbf{1}\leq \boldsymbol{\tau} \leq \mathbf{b}^{(\mathbf{\smallT})}(\mathbf{c})$, $\mu = ivec(\boldsymbol\mu,\mathbf{b}^{(\mathbf{\smallS})}(\mathbf{r}))$ and $\tau = ivec(\boldsymbol\tau,\mathbf{b}^{(\mathbf{\smallT})}(\mathbf{c}))$,
then block $(\mu,\tau)$ of ${\cal H}_{{\mathbf{\smallR}}\times{\mathbf{\smallC}}}$ is given by
\[
\left( {\cal H}_{{\mathbf{\smallR}}\times{\mathbf{\smallC}}}\right)_{\boldsymbol{\mu},\boldsymbol{\tau}}
\;=\;\sum_{\mathbf{q}=\mathbf{1}}^{\mathbf{b}^{(\mathbf{\smallS})}(\boldsymbol{\lambda})}
\left( {\cal F}_{\mathbf{\smallR}\times\mathbf{\smallLambda}}\right)_{\boldsymbol{\mu},\mathbf{q}}
\left( {\cal G}_{\mathbf{\smallPsi}\times\mathbf{\smallC}}\right)_{\mathbf{q},\boldsymbol{\tau}}.
\]
Using (3.17) this can be rewritten in terms of subtensor unfoldings. Indeed, if
index vectors $\mathbf{k}$, $\mathbf{i}^{(\mathbf{q})}$, and $\mathbf{j}^{(\mathbf{q})}$
are defined by
\[
\begin{array}{lclclcl}
\mathbf{k}(\mathbf{r}) \!\!&\!\!=\!\!&\!\! \boldsymbol{\mu} &\quad & \mathbf{k}(\mathbf{c}) \!\!&\!\!=\!\!&\!\! \boldsymbol{\tau}\\
\mathbf{i}^{(\mathbf{q})}(\mathbf{r}) \!\!&\!\!=\!\!&\!\! \mathbf{k}(\mathbf{r}) &\quad &
\mathbf{i}^{(\mathbf{q})}(\boldsymbol{\lambda}) \!\!&\!\!=\!\!&\!\! \mathbf{q}\rule{0pt}{13pt}\\
\mathbf{j}^{(\mathbf{q})}(\boldsymbol{\psi}) \!\!&\!\!=\!\!&\!\! \mathbf{q} &\quad &
\mathbf{j}^{(\mathbf{q})}(\mathbf{c}) \!\!&\!\!=\!\!&\!\! \mathbf{k}(\mathbf{c}) \rule{0pt}{13pt}
\end{array}
\]
then
\begin{equation}\label{blockContr}
 \left({\cal H}_{\mathbf{k}}\right)_{[1:f]\times[f+1:f+g]}
\;=\;\sum_{\mathbf{q}=\mathbf{1}}^{\mathbf{b}^{(\mathbf{\smallS})}(\boldsymbol{\lambda})}
\left( {\cal F}_{\mathbf{i}^{\mathbf{(q})}}\right)_{\mathbf{r}\times\boldsymbol{\lambda}}
\left( {\cal G}_{\mathbf{j}^{\mathbf{(q})}}\right)_{\boldsymbol{\psi}\times\mathbf{c}}.
\end{equation}

\subsection{Blocked Multilinear Products}

As an example of how the preceding results can be adapted to handle structured contractions, we briefly consider the multilinear product since we have developed the supporting formulae in \S2.4 and \S3.4.  Suppose \inTens{A}{}{n_{1}}{n_{{d}}} and that 
\[
\inm{B^{(k)}}{q_{k}}{n_{k}} \qquad k=1,\ldots,{d}.
\]
The tensor \inTens{C}{}{q_{1}}{q_{{d}}} specified by
\begin{equation}
{\cal C}(\mathbf{i}) \;=\; \sum_{\mathbf{k}=\mathbf{1}}^{\mathbf{n}}
{\cal A}(\mathbf{k})
B^{(1)}(i_{1},k_{1})\cdots B^{({d}\,)}(i_{d},k_{{d}})
\end{equation}
is the {\em multilinear product} of $\cal A$ with $B^{(1)},\ldots,B^{({d}\,)}$
and is denoted \cite{lim} by
\[
{\cal C} \;=\; (B^{(1)},\ldots,B^{({d}\,)})\!\cdot\! {\cal A}.
\]
If the order-$(2{d})$ tensor $\cal B$ is defined by
\[
{\cal B} \;=\; B^{(1)} \circ \cdots \circ B^{(d)},
\]
then we see that $\cal C$ is a contraction of the form
\[
{\cal C}(\mathbf{i}) \;=\; \sum_{\mathbf{k}=\mathbf{1}}^{\mathbf{n}}
{\cal A}(\mathbf{k}){\cal B}(i_{1},k_{1},\ldots,i_{{d}},k_{{d}}).
\]
We apply Theorem 4.1 with ${\cal F} = {\cal B}$, $f = d$, $\ell = d$, ${\cal G} = {\cal A}$,
$g = 0$, $\mathbf{r} = 1:2:2d$, ${\boldlambda} = 2:2:2d$, 
$\boldsymbol{\psi} = 1:d,$ and $ \mathbf{c} = \emptyset$. It follows that
${\cal A}_{\boldsymbol{\psi}\times\mathbf{c}} = \mbox{vec}({\cal A})$ and
${\cal C}_{[1:\ell]\times [\ell+1:\ell]} = \mbox{vec}({\cal C})$
and so
 from Theorem 4.1 and (2.25) we have
\begin{equation}\label{bmult1}
\mbox{vec}({\cal C}) \;=\; \left( B^{(d)} \T \cdots \T B^{(1)}\right) \mbox{vec}({\cal A}).
\end{equation}
If the $B$ matrices are blocked according to (3.21) and $\mathbf{R}$ and $\mathbf{C}$ are defined by (3.22)-(3.23), then $\mathbf{R}$ is a blocking for $\cal C$, $\mathbf{C}$ is a blocking for $\cal A$, and
\begin{equation}
P_{{\mathbf{\smallR}}}\mbox{vec}({\cal C}) \;=\; \left(P_{\mathbf{\smallR}} \left( B^{(d)} \T \cdots \T B^{(1)}\right)P_{\mathbf{\smallC}}^{T}\right) P_{\mathbf{\smallC}} \mbox{vec}({\cal A}).
\end{equation}
From (3.24) we see that the matrix
\begin{equation}\label{bmult2}
{\cal B}_{\mathbf{\smallR}\times\mathbf{\smallC}} \;=\;
P_{\mathbf{\smallR}} \left( B^{(d)} \T \cdots \T B^{(1)}\right)P_{\mathbf{\smallC}}^{T}
\end{equation}
is a block matrix whose entries are Kronecker products. 
Indeed, 
${\cal B}_{\mathbf{\smallR}\times\mathbf{\smallC}}$
 is  essentially the Tracy-Singh product of the $B$-matrices, see \cite{Tracy}.
Thus, from \eqref{bmult1}-\eqref{bmult2} we have the following block specification for $\cal C$:
\begin{equation}\label{blockedmulti}
\mbox{vec}_{{{\mathbf{\smallR}}}}({\cal C}) \;=\; {\cal B}_{\mathbf{\smallR}\times\mathbf{\smallC}}
\mbox{vec}_{\mathbf{\smallC}}({\cal A}).
\end{equation}

\subsection{Visualization}
As in block matrix computations, it is sometimes important to view
a given blocked tensor contraction from different viewpoints. A small example
builds an appreciation for this point.

Suppose $\cal F$ is a $3 \times 4 \times 2$ block tensor and $\cal G$ is a $2 \times 3 \times 5$ block tensor such that the blockings in mode 3 in $\cal F$ and mode 1 in $\cal G$ conform. Let $\cal H$ be the $3 \times 4 \times 3 \times 5$ block tensor whose elements are given by
\[
\mathcal{H} (i_1,i_2,j_1,j_2) = \sum_{k} \mathcal{F}(i_1,i_2,k)\cdot \mathcal{G} (k,j_1,j_2).
\]
For convenience, denote the operation of contracting two order-3 tensors $\mathcal{T}_1$ and $\mathcal{T}_2$ in this way  as $\mathcal{T}_1\star \mathcal{T}_2$,
e.g., ${\cal H} = {\cal F} \star {\cal G}$.
Fig 4.1 shows how this blocked contraction can be visualized at three different levels.
At the lowest level,
block $[a,b,c,d]$ in $\cal H$ can be computed via the matrix equation
\[
(\mathcal{H}_{abcd})_{[1\; 2]\times[3\;4]} = (\mathcal{F}_{ab 1})_{[1\; 2] \times [3]} \cdot (\mathcal{G}_{1 cd})_{[1]\times [2\; 3]} + (\mathcal{F}_{ab 2})_{[1\; 2] \times [3]} \cdot (\mathcal{G}_{2 cd})_{[1]\times [2\; 3]}.
\]
 This follows from \eqref{blockContr} and is depicted in part (3) of Fig 4.1. 

\begin{figure}
\begin{tabular}{V  p{6.2cm}}
\scalebox{.65}{
\begin{picture}(260,120)(-10,0)
\put(125,30){\makebox(0,0){\Huge $\star$}}


\put(10,-10){\thicklines \line(1,0){80}}
\put(10,10){\thicklines \line(1,0){80}}
\put(10,30){\thicklines \line(1,0){80}}
\put(10,50){\thicklines \line(1,0){80}}
\put(10,-10){\thicklines \line(0,1){60}}
\put(30,-10){\thicklines \line(0,1){60}}
\put(50,-10){\thicklines \line(0,1){60}}
\put(70,-10){\thicklines \line(0,1){60}}
\put(90,-10){\thicklines \line(0,1){60}}
\put(10,50){\thicklines \line(1,1){20}}
\put(30,50){\thicklines \line(1,1){20}}
\put(50,50){\thicklines \line(1,1){20}}
\put(70,50){\thicklines \line(1,1){20}}
\put(90,50){\thicklines \line(1,1){20}}
\put(20,60){\thicklines \line(1,0){80}}
\put(30,70){\thicklines \line(1,0){80}}
\put(90,50){\thicklines \line(1,1){20}}
\put(90,30){\thicklines \line(1,1){20}}
\put(90,10){\thicklines \line(1,1){20}}
\put(90,-10){\thicklines \line(1,1){20}}
\put(100,0){\thicklines \line(0,1){60}}
\put(110,10){\thicklines \line(0,1){60}}

\put(140,-10){\thicklines \line(1,0){60}}
\put(140,10){\thicklines \line(1,0){60}}
\put(140,30){\thicklines \line(1,0){60}}
\put(140,-10){\thicklines \line(0,1){40}}
\put(160,-10){\thicklines \line(0,1){40}}
\put(180,-10){\thicklines \line(0,1){40}}
\put(200,-10){\thicklines \line(0,1){40}}
\put(140,30){\thicklines \line(1,1){50}}
\put(160,30){\thicklines \line(1,1){50}}
\put(180,30){\thicklines \line(1,1){50}}
\put(200,30){\thicklines \line(1,1){50}}
\put(200,30){\thicklines \line(1,1){50}}
\put(200,10){\thicklines \line(1,1){50}}
\put(200,-10){\thicklines \line(1,1){50}}
\put(210,0){\thicklines \line(0,1){40}}
\put(220,10){\thicklines \line(0,1){40}}
\put(230,20){\thicklines \line(0,1){40}}
\put(240,30){\thicklines \line(0,1){40}}
\put(250,40){\thicklines \line(0,1){40}}
\put(150,40){\thicklines \line(1,0){60}}
\put(160,50){\thicklines \line(1,0){60}}
\put(170,60){\thicklines \line(1,0){60}}
\put(180,70){\thicklines \line(1,0){60}}
\put(190,80){\thicklines \line(1,0){60}}

\end{picture}}

& {\footnotesize $(1)$ The tensor contraction $\mathcal{H} = \mathcal{F} \star \mathcal{G}$ of two order-3 tensors viewed graphically as a contraction of conformally blocked tensors.} \\

\scalebox{.65}{
\begin{picture}(190,80)(-10,-10)


\put(105,-10){\makebox(0,0){\Huge $\star$}}


\put(10,-60){\thicklines \line(1,0){40}}
\put(10,-20){\thicklines \line(1,0){40}}
\put(30,0){\thicklines \line(1,0){40}}
\put(50,20){\thicklines \line(1,0){40}}
\put(10,-60){\thicklines \line(0,1){40}}
\put(50,-60){\thicklines \line(0,1){40}}
\put(50,-20){\thicklines \line(1,1){40}}
\put(50,-60){\thicklines \line(1,1){40}}
\put(10,-20){\thicklines \line(1,1){40}}
\put(70,-40){\thicklines \line(0,1){40}}
\put(90,-20){\thicklines \line(0,1){40}}

\put(120,-60){\thicklines \line(1,0){40}}
\put(120,-20){\thicklines \line(1,0){40}}
\put(120,20){\thicklines \line(1,0){40}}
\put(120,-60){\thicklines \line(0,1){80}}
\put(160,-60){\thicklines \line(0,1){80}}
\put(120,20){\thicklines \line(1,1){20}}
\put(160,20){\thicklines \line(1,1){20}}
\put(160,20){\thicklines \line(1,1){20}}
\put(160,-20){\thicklines \line(1,1){20}}
\put(160,-60){\thicklines \line(1,1){20}}
\put(180,-40){\thicklines \line(0,1){80}}
\put(140,40){\thicklines \line(1,0){40}}

\end{picture}}

&  {\footnotesize $(2)$ Block $\mathcal{H}_{abcd}=\mathcal{H}(\alpha_1:\alpha_2,\beta_1:\beta_2,\gamma_1:\gamma_2,\delta_1:\delta_2)$ is a $\star$-contraction of two ``block fibers'', one from $\cal F$ and one from $\cal G$, i.e.
$\mathcal{H}_{abcd} = \mathcal{F}(\alpha_1:\alpha_2,\beta_1:\beta_2,:)\star \mathcal{G}(:,\gamma_1:\gamma_2,\delta_1:\delta_2). $ }
 \\

\scalebox{.45}{
\begin{picture}(370,160)(-10,0)


\put(85,40){\makebox(0,0){\Huge \scalebox{1.5}{$\star$}}}
\put(285,40){\makebox(0,0){\Huge \scalebox{1.5}{$\star$}}}
\put(185,40){\makebox(0,0){\Huge $+$}}


\put(10,10){\thicklines \line(1,0){40}}
\put(10,50){\thicklines \line(1,0){40}}
\put(30,70){\thicklines \line(1,0){40}}
\put(10,10){\thicklines \line(0,1){40}}
\put(50,10){\thicklines \line(0,1){40}}
\put(50,50){\thicklines \line(1,1){20}}
\put(50,10){\thicklines \line(1,1){20}}
\put(10,50){\thicklines \line(1,1){20}}
\put(70,30){\thicklines \line(0,1){40}}

\put(100,10){\thicklines \line(1,0){40}}
\put(100,50){\thicklines \line(1,0){40}}
\put(120,70){\thicklines \line(1,0){40}}
\put(100,10){\thicklines \line(0,1){40}}
\put(140,10){\thicklines \line(0,1){40}}
\put(140,50){\thicklines \line(1,1){20}}
\put(140,10){\thicklines \line(1,1){20}}
\put(100,50){\thicklines \line(1,1){20}}
\put(160,30){\thicklines \line(0,1){40}}


\put(210,10){\thicklines \line(1,0){40}}
\put(210,50){\thicklines \line(1,0){40}}
\put(230,70){\thicklines \line(1,0){40}}
\put(210,10){\thicklines \line(0,1){40}}
\put(250,10){\thicklines \line(0,1){40}}
\put(250,50){\thicklines \line(1,1){20}}
\put(250,10){\thicklines \line(1,1){20}}
\put(210,50){\thicklines \line(1,1){20}}
\put(270,30){\thicklines \line(0,1){40}}

\put(300,10){\thicklines \line(1,0){40}}
\put(300,50){\thicklines \line(1,0){40}}
\put(320,70){\thicklines \line(1,0){40}}
\put(300,10){\thicklines \line(0,1){40}}
\put(340,10){\thicklines \line(0,1){40}}
\put(340,50){\thicklines \line(1,1){20}}
\put(340,10){\thicklines \line(1,1){20}}
\put(300,50){\thicklines \line(1,1){20}}
\put(360,30){\thicklines \line(0,1){40}}

\end{picture}}
&
{\footnotesize $(3)$ The $\star$-contraction of the two block fibers is  a sum of $\star$-contractions of fiber blocks, i.e.~$\mathcal{H}_{abcd} = \mathcal{F}_{ab 1}\star \mathcal{G}_{1 c d} + \mathcal{F}_{ab 2}\star \mathcal{G}_{2 c d}.$}
\end{tabular}

\caption{Three Levels of a Blocked Contraction}

\end{figure}
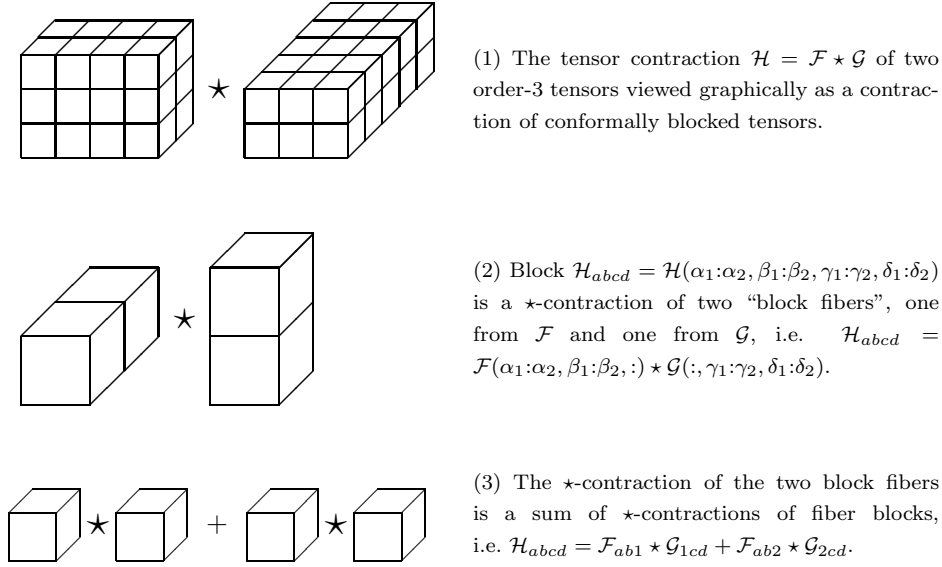

\section{Concluding Remarks}

Given the nature of this paper, it is important to be reminded in this closing section that there is a big difference
between a cryptic mathematical formula and its utilization in practice. A case in point is
the permutation matrix $P_{\mathbf{\smallM}}$ that is characterized in Theorem 3.1. Obviously, an integer vector should  be  used to represent a permutation matrix like $P_{\mathbf{\smallM}}$; it should never be computed as a two-dimensional array. We offer a few details based on the convention that
if
$P = I_{n}(:,\mathbf{v})$ where $\mathbf{v}$ is permutation of $1:n$, then $\mathbf{v}$ represents $P$.
We capture this connection with the notation $P_{\mathbf{v}}$. Note  that if $y = P_{\mathbf{v}}x$, then
$y = x(\mathbf{v})$ while $y(\mathbf{v}) = x$ implies $y = P_{\mathbf{v}}^{T}x$. Letting
$\mathbf{1}_{n}$ denote the $n$-vector of ones, here are some basic facts that concern this style of representation:
\begin{enumerate}
\item If $q$ and $r$ are positive integers and
$\mathbf{w} = [1:r:qr \; 2:2:qr \;\cdots \; r:r:qr]$,
then $P_{\mathbf{w}} \:=\: \Pi_{q,r}$, the $(q,r)$ perfect shuffle.
\item If $\mathbf{u}$ and $\mathbf{v}$ are permutations of $1:n$ and 
$\mathbf{w} = \mathbf{v}(\mathbf{u})$, then
$P_{\mathbf{w}} = P_{\mathbf{u}}P_{\mathbf{v}}$. 
\item
If $\mathbf{u}$ is a permutation of $1:n$ and $\mathbf{v}$ is a permutation of $1:m$,
then $P_{\mathbf{w}} = P_{\mathbf{u}} \T P_{\mathbf{v}}$ where
$\mathbf{w} = \mathbf{1}_{n} \T \mathbf{v} \:+\: m\cdot(\mathbf{u} - \mathbf{1}_{n}) \T \mathbf{1}_{m}$.
\item If $\mathbf{u}$ is a permutation of $1:n$ and $\mathbf{v}$ is a permutation of $1:m$,
then $P_{\mathbf{w}} = \mbox{diag}(P_{\mathbf{u}},P_{\mathbf{v}})$ where $\mathbf{w} = [\mathbf{u}\; \;(n\!\cdot\!\mathbf{1}_{n}+\mathbf{v})]$
\end{enumerate}
 The vector representation of the matrix $P_{\mathbf{\smallM}}$, since it is defined by perfect shuffles,
Kronecker products, and direct sums, can be efficiently assembled using these facts.

Another illustration of the gap between formula and implementation concerns equation (4.11).
The calculation of a multilinear product ${\cal C} = (B^{(1)},\ldots,B^{(d)})\cdot {\cal A}$ would not explicitly use this formula. Instead it would proceed as follows:

\medskip

\begin{code}
\> {\bf for} $i = 1,\ldots,d$\\
\>\> ${\cal A} \; \leftarrow \; (I_{n_1},\ldots,B^{(i)},\ldots,I_{n_d})\!\cdot\! {\cal A}\rule{0pt}{12pt}$\\
\> {\bf end}
\end{code}
The $i$-th update is referred to as the \emph{$i$-mode product}, see \cite{HOSVD,TammySirev}. 
By using Theorem 4.1 we see that this is equivalent to the matrix-matrix multiplication
\[
{\cal A}_{(i)}  \; \leftarrow \; B^{(i)} {\cal A}_{(i)}
\]
where ${\cal A}_{(i)}  \equiv  {\cal A}_{[ i ] \times [1:i-1 \; i+1:d]}$ is the mode-$i$ unfolding of $\cal A$ mentioned in \S2.3. 

Similarly, in a block-based implementation of the multilinear product, one would not directly use \eqref{blockedmulti}. Instead, the block-matrix multiplications
\[
{\cal A}_{\mathbf{\smallJ}\times \mathbf{\smallC}} \; \leftarrow \; B^{(i)} {\cal A}_{\mathbf{\smallI}\times\mathbf{\smallC}}
\]
would be carried out sequentially for modes $i=1,\ldots,d$. Here, $\mathbf{I}$ is the original blocking for mode $i$, $\mathbf{J}$ is the new blocking of mode $i$ inherited from the row blocking of $B^{(i)}$, and $\mathbf{C}$ is a blocking for modes $[1:i\!-\!1 \;\, i\!+\!1:d]$ of $\cal A$.

Overall, it is reasonable to conclude from the above that block tensors behave  in much the
same way as block matrices. Although the precise formulas are more involved, the basic intuition that ``all operations can be done at the block level'' is correct. 
By making precise the notion of a block unfolding and developing a  framework for reasoning about
block tensor computation, we hope that we have laid a modest foundation for further research.
Our own agenda includes looking at block versions of the tensor contraction engine
\cite{TCE}, developing recursive tensor data structures that extend the clever ideas in 
\cite{RecurBlock}, expanding the functionality of the  Tensor Toolbox \cite{koldaToolbox,koldaSoftware} so that it
supports block tensor computation, and analyzing block versions of various
 tensor iterations such as \cite{power}. Throughout all this it will be important to
chip away at the ``notational divide'' that currently besets the tensor computation community,
see \cite{Kiers}.

\bigskip

{\bf Acknowledgements.}  The authors are indebted to the two referees whose suggestions and corrections led to major improvements over the first and second  drafts of the manuscript.

\end{document}